\documentclass[article,english]{imsart}
\usepackage[T1]{fontenc}
\usepackage[latin9]{inputenc}
\usepackage{geometry}
\usepackage{float}
\usepackage{amstext}
\usepackage{amsthm}
\usepackage{amsmath}
\usepackage{amssymb}
\usepackage{graphicx}

\makeatletter

\floatstyle{ruled}
\newfloat{algorithm}{tbp}{loa}
\providecommand{\algorithmname}{Algorithm}
\floatname{algorithm}{\protect\algorithmname}

  \theoremstyle{definition}
  \newtheorem{defn}{\protect\definitionname}
  \theoremstyle{remark}
  \newtheorem{rem}{\protect\remarkname}
  \theoremstyle{plain}
  \newtheorem{lem}{\protect\lemmaname}
\theoremstyle{plain}
\newtheorem{thm}{\protect\theoremname}
  \theoremstyle{plain}
  \newtheorem{prop}{\protect\propositionname}
  \theoremstyle{plain}
  \newtheorem{cor}{\protect\corollaryname}

\usepackage{algorithm,algorithmic}
\usepackage[authoryear]{natbib}
\date{}
\usepackage{stmaryrd}

\usepackage{mathabx}
\newcounter{hypA}
\newenvironment{hyp}{\refstepcounter{hypA}\begin{itemize}
\item[({\bf A\arabic{hypA}})]}{\end{itemize}}

\@ifundefined{showcaptionsetup}{}{%
 \PassOptionsToPackage{caption=false}{subfig}}
\usepackage{subfig}
\makeatother

\usepackage{babel}
  \providecommand{\definitionname}{Definition}
  \providecommand{\lemmaname}{Lemma}
  \providecommand{\propositionname}{Proposition}
  \providecommand{\remarkname}{Remark}
\providecommand{\corollaryname}{Corollary}
\providecommand{\theoremname}{Theorem}

\begin{document}
\begin{frontmatter}

\title{Probably approximate Bayesian computation: nonasymptotic convergence of ABC under mispecification}

\runtitle{Probably approximate Bayesian computation}
\begin{aug}

\author{\fnms{James} \snm{Ridgway}\ead[label=e1]{james.lp.ridgway@gmail.com}}

\address{Capital Fund Management}

\runauthor{J. Ridgway}
\affiliation{CFM}
\end{aug}

\begin{abstract}
Approximate Bayesian computation (ABC) is a widely used inference
method in Bayesian statistics to bypass the point-wise computation
of the likelihood. In this paper we develop theoretical bounds for
the distance between the statistics used in ABC. We show that some
versions of ABC are inherently robust to mispecification. The bounds
are given in the form of oracle inequalities for a finite sample size.
The dependence on the dimension of the parameter space and the number
of statistics is made explicit. The results are shown to be amenable
to oracle inequalities in parameter space. We apply our theoretical
results to given prior distributions and data generating processes,
including a non-parametric regression model. In a second part of the
paper, we propose a sequential Monte Carlo (SMC) to sample from the
pseudo-posterior, improving upon the state of the art samplers. 
\end{abstract}

\tableofcontents

\end{frontmatter}

\section{Introduction }

A wide range of statistical applications involve models where the
likelihood is not available in closed form. One typical case is when
the likelihood is expressed as a multidimensional integral, as in
state space models or models with high dimensional latent variables.
Another well studied case appears when the normalizing constant of
the likelihood is unknown, that is, the likelihood $\pi_{\theta}$
is written as $\pi_{\theta}(x)=\ell_{\theta}(x)/Z_{\theta}$, with
unknown $Z_{\theta}$. More generally we consider in this paper models
with hard to compute likelihoods that are relatively easy to sample
from. We refer the reader to \cite{Marin2012} for further motivation
for this framework. 

The goal of Approximate Bayesian Computation (ABC) is to perform statistical
inference in the case where the likelihood cannot be evaluated point-wise.
This algorithm has found success amongst applied statisticians in
fields as diverse as population genetics (\cite{beaumont2002approximate}),
astronomy (\cite{cameron2012approximate}), computer vision (\cite{mansinghka2013approximate})
etc. The general idea is to use auxiliary samples from the model for
different values of the parameter and compare them with the observed
variables. More precisely for a sequence of observations $(Y_{i})_{i=1}^{n}$,
and a vector of summary statistics $S$, ABC can be formulated as sampling
from the joint distribution 
\begin{equation}
\pi_{h}^{abc}\left(d\theta,dX^{n}\vert Y^{n}\right)=\frac{K_{h}\left\{ d\left(S(X^{n}),S(Y^{n})\right)\right\} \pi_{\theta}(dX^{n})\pi(d\theta)}{\int\int K_{h}\left\{ d\left(S(X^{n}),S(Y^{n})\right)\right\} \pi_{\theta}(dX^{n})\pi(d\theta)},\label{eq:Kernel-ABC}
\end{equation}
where $K_{h}$ is a kernel function with window $h$,
and $d$ is a distance between summary statistics. Inference in parameter
space is obtained by integrating in in $X$. The intuition is to sample
$\theta$ in the parameter space and an auxiliary sample $(X_{i})_{i=1}^{n}$
according to the model, such that the $(X_{i})_{i=1}^{n}$ are ``close''
to the observation. The kernel function enforces the closeness between
the simulated data and the observations. Several kernels are proposed
in practice, however two are widely used: the uniform kernel $\textbf{1}_{d(S(X^{n}),S(Y^{n}))<h}$
(corresponding to an accept reject algorithm) and the Gaussian kernel
$\frac{1}{\sqrt{h}}e^{-\frac{1}{h}d^{2}(S(X^{n}),S(Y^{n}))}$. 

Several points are usually discussed in the burgeoning literature
of ABC: the choice of summary statistics (\cite{Fearnhead2012,marin2014relevant}),
convergence of the Monte Carlo algorithm as $h$ goes to $0$
(e.g. \cite{barber2015rate}) and statistical properties of the pseudo-posterior
(\cite{li2015asymptotic,frazier2016}). We focus in this paper on
the latter point, the statistics are considered fixed and we let the
window $h$ go to $0$ only as a function of the sample size.
Taking this approach rather than trying to let $h$ go to zero
with the size of the Monte Carlo sample allows an appropriate treatment
of the bias introduced by ABC. 

In general the validity of the pseudo-posterior is given intuitively
by saying for $h\rightarrow0$ the pseudo-posterior ``should''
converge to the ``ideal'' one (Bayesian posterior). Relatively few
papers consider the statistical error induced by using distribution
(\ref{eq:Kernel-ABC}), and its rate of convergence. Notable exceptions
are given in \cite{frazier2016} in the uniform kernel case and \cite{li2015asymptotic}
under moment conditions on the kernel. The former studies posterior
consistency and its concentration rate as $n\rightarrow\infty$. The
latter gives a central limit theorem for the posterior mean of ABC.
Both papers study asymptotic convergence for models that are well
specified. We can also mention the recent paper of \cite{bernton2017inference}
where some results in Wasserstein distance for mispecified models
are adapted from \cite{frazier2016}. We also mention that since the first 
submission of our paper \cite{Frazier2017} proposed an extention of \cite{frazier2016}
to the misspecified case. The paper proposes an asymptotic analysis 
of the ABC error under misspecification. The conclusion on the robustness 
of ABC and the choice of the window parameter are similar. The main differences lies in the 
choice of the kernel $K_h$ (with empirical improvement), the ability to give non asymptotic results 
and to deal with high dimension parameters.

The form of the ABC pseudo-posterior suggests that it should allow
some degree of mispecification, in a spirit similar to generalized
posteriors (e.g. \cite{zhang2006}). We call generalized posterior,
for $\lambda>0$, a measure of the form 
\[
\pi_{\lambda}^{gen}\left(d\theta\vert Y^{n}\right)\propto\left\{ \pi_{\theta}\left(Y^{n}\right)\right\} ^{\lambda}\pi(d\theta).
\]
 The case $\lambda=1$ corresponds to the usual Bayesian posterior
distribution. The case where $\lambda<1$ is used to deal with misspecification
in \cite{zhang2006} and \cite{Gruenwald2014}. This works intuitively
by assigning less weight to the likelihood part. In this spirit ABC
seems to be built for the same kind of robustness. Similarly allowing
for a bigger window (or one converging less quickly) allows for some
robustness in the approach. In this paper we propose studying the
convergence of the pseudo-posterior in the case where the true distribution
of the sample lies outside the statistician's model. We rely on concentration
inequalities for the distance under general distributions. Bounds
are given in deviation using PAC-Bayesian analysis (e.g. \cite{Catoni2007}). 

Our results will rely on the use of an exponential kernel in equation
(\ref{eq:Kernel-ABC}). This also has some computational advantages
as compared to a uniform kernel. It allows to build a smooth sequence
of distribution indexed by the inverse window parameter. We will propose
a new Sequential Monte Carlo (SMC) algorithm to efficiently explore
this sequence. We build a fully adaptive version of the algorithm
of \cite{DelMoral2012} .

In the next section we define the framework that we will be using
in the paper. In Section \ref{sec:Theoretical-bounds} we give oracle
inequalities on the distance between the expected statistics. We also
give some bounds on the parameters themselves under additional assumptions.
In Section \ref{subsec:Empirical-bounds} we take a brief detour to
develop empirical bounds. Those results can be used in practice to
bound the generalization error (expected distance). We use those bounds
to develop an estimator that chooses automatically the bandwidth parameter
(Section \ref{subsec:Adaptive-bounds}). Section \ref{sec:Examples-of-bounds}
is devoted to some application of the bounds to different priors and
models. 

Finally in Section \ref{sec:Monte-Carlo-algorithm} we build upon
the sequential Monte Carlo sampler of \cite{DelMoral2012} to give
an efficient way to sample from the joint pseudo-posterior. The algorithm
is applied to a toy example in a numerical section (Section \ref{sec:Numerical-experiments}).
The proofs are deferred to Section \ref{sec:Proofs-and-supporting}.

\subsubsection*{Notation}

In the following we will use operator notation i.e. for a finite measure
$\nu$ we write for any $f\in L^{1}(\nu)$ the expectation $\nu(f):=\int f(x)\nu(dx).$
The support of measure $\nu$ is denoted $\text{supp}(\nu)$, and
vectors $Y_{n}^{m}:=(Y_{n},\cdots,Y_{m})$ for any $n<m$, we will
drop the subscript when the vector starts at $1$; we also use $n:m:=\left(n,\cdots,m\right)$.
For some set $A$ we write $\mathcal{M}_{1}^{+}(A)$ the set of probability
measures on $A$, the $\sigma$-algebra is given by context. We use
$\wedge$ and $\vee$ for the $\inf$ and $\sup$ respectively and
$v_{n}\asymp r_{n}$ for two sequences such that $v_{n}/r_{n}=\mathcal{O}(1)$.
The notation $\sup_{\pi}$ for $\pi\in\mathcal{M}_{+}^{1}$
is intended to mean the supremum over the support of $\pi$. Furthermore 
for two measures $\mu,\nu$ we write $\mu\otimes\nu$ for the product measure and 
$\mathcal{K}(\mu,\nu)$ for the Kullback-Leibler divergence.

\section{Set up and definitions}

Define a probability space $(\mathcal{Y}^{n},\mathcal{B}(\mathcal{Y}^{n}),\mathbb{P})$,
we suppose that the observation $Y^{n}\in\mathcal{Y}^{n}$ is an i.i.d.
collection sampled from the probability measure $\mathbb{P}$. We
define a model as a collection of measures $\left(\pi_{\theta},\theta\in\Theta\right)$
on some space $\left(\mathcal{X}^{n},\mathcal{B}(\mathcal{X}^{n})\right)$
indexed by a parameter $\theta\in\Theta$. We let $(\Theta,\Vert.\Vert)$
be normed and endow $\Theta$ with the structure of a probability
space $\left(\Theta,\mathcal{B}(\Theta),\pi\right)$ where $\pi$
is the prior probability. 

For the moment we do not assume anything on the probability $\mathbb{P}$
other than i.i.d., furthermore this hypothesis could also be weakened.
We will see that we require only some form of exponential concentration
inequality for the probability, this might also be obtained for more
general assumptions such as weak-dependance (e.g. \cite{olivier2010deviation}).
To shorten formulae in the text we also define the following marginal
measure $m_{X}(A):=\int_{A}\int_{\Theta}\pi(d\theta)\pi_{\theta}(dX^{n})$
for any $A\in\mathcal{B}(\mathcal{X}^{n})$. We will abuse notation
and write $\pi_{\theta}\pi(dX^{n},d\theta)$ for $\pi(d\theta)\pi_{\theta}(dX^{n})$,
hence $m_{X}(f)=\pi_{\theta}\pi(f)$ for any suitable function $f$
on $\mathcal{X}^{n}$.

We can now build the joint pseudo-posterior of interest. We start
by defining the summary statistics. Let $(\mathcal{S},d)$ be a  metric
space, the summary statistics is a function $S:\mathcal{X}^{n}\cup\mathcal{Y}^{n}\mapsto\mathcal{S}$.
The metric on $\mathcal{S}$ is the one we will use to compare samples
in the pseudo posterior. Notice that we do not assume that the samples
and the model live on the same sample space but rather that the summary
statistics maps both samples to $\mathcal{S}$. In the paper we will
consider cases where the observations are sampled on a bounded domain
but the auxiliary sample can reach out of this set. The choice of
the statistics has an impact on the quality of the approximation.
The ideal case in ABC is the one where the statistics are exhaustive
and the map $\theta\mapsto\pi_{\theta}(S)$ is injective (see \cite{frazier2016}).
Finding a set of statistics that correctly summarizes the distribution
is of course hard, if not impossible. In particular we do not have
access to the true distribution and can only hope in finding a summary
statistics for our model. We will not discuss further the choice of
the summary statistics as our measure of risk will itself depend on
the quality of this choice.

Another degree of freedom lies in the choice of the distance. This
has not been studied much in the literature. We will see that it arises
in our measure of the risk associated with distribution and different
dimension dependence. In the examples of Section \ref{sec:Examples-of-bounds}
we give two possible distances for which we can give theoretical results.
To simplify notations further, when it is necessary, we will write
$D_{n}^{S}$ for the distance $d(S(X^{n}),S(Y^{n}))$, where $Y^{n}$
is the observed vector and $X^{n}$ is an auxiliary sample from the
model. In what follows we will call $X^{n}$ the prior sample and
$\left(\pi_{\theta}\right)_{\theta\in\Theta}$ the prior model, as
this choice does not (formally) depend on the data and is confronted
to samples from a probability $\mathbb{P}$. 

The version of the ABC pseudo-posterior used in this paper can now
be defined. Using the previous notation we let
\begin{align*}
\varrho_{\lambda}(d\theta,dX^{n})= & \frac{1}{Z_{\lambda,\pi}}e^{-\lambda d(S(X^{n}),S(Y^{n}))}\pi_{\theta}(dX^{n})\pi(d\theta),\\
\text{where}\quad & Z_{\lambda,\pi}=\int_{\mathcal{X}^{n}\times\Theta}e^{-\lambda d(S(X^{n}),S(Y^{n}))}\pi_{\theta}(dX^{n})\pi(d\theta).
\end{align*}

This can be seen as the kernel-ABC defined in introduction (equation
(\ref{eq:Kernel-ABC})) with an exponential kernel. We have replaced
the window $\epsilon$ by its inverse $\lambda$. By analogy with
statistical Physics, when using a exponential kernel, we will refer
to $\epsilon$ as the temperature and $\lambda$ as the inverse temperature.
We define a set $\mathcal{I}\subset\mathbb{R}_{+}$ and let $\lambda\in\mathcal{I}$.
The main interest of end users of ABC is the marginal in $\theta$
of this joint distribution. We abuse notation and give the following
definition of ABC
\begin{defn}
\label{def:abc-theta}We define the ABC pseudo-posterior as the following
distribution

\[
\varrho_{\lambda}(d\theta):=\frac{\int_{\mathcal{X}^{n}}e^{-\lambda d(S(X^{n}),S(Y^{n}))}\pi_{\theta}(dX^{n})\pi(d\theta)}{\int_{\Theta}\int_{\mathcal{X}^{n}}e^{-\lambda d(S(X^{n}),S(Y^{n}))}\pi_{\theta}(dX^{n})\pi(d\theta)},
\]
 for any $\lambda\in\mathcal{I}$.

We define furthermore the following marginal in $X^{n}$, $\ensuremath{\varrho_{X,\lambda}(dX^{n}):=\int_{\Theta}\varrho_{\lambda}(dX^{n},d\theta)}.$
\end{defn}
As was emphasized in the introduction the pseudo-posterior is built
by augmenting the sample space with latent variables representing
a sample from prior model $\pi_{\theta}$. Those samples are weighted
according to the exponential weights $e^{-\lambda d(S(X^{n}),S(Y^{n}))}$,
thus the samples that are close to the observations receive higher
weight than those that are far. The inverse temperature parameter
$\lambda$ emphasizes the spikiness of the weights as it increases.
The marginal $\rho_{X,\lambda}$ is an important quantity in our study.
We expect that it should, for large enough sample size, be close to
samples the ``best model''. Because we treat the case where
$\mathbb{P}\notin\left\{ \pi_{\theta},\theta\in\Theta\right\} $ the
best $\theta$ may appear as a vague notion.
We will define the latter as the oracle parameter,
\begin{defn}
\label{def:The-oracle-parameter}The oracle parameter $\theta^{\star}$
is given by 
\[
\theta^{\star}\in\arg\min_{\theta\in\Theta}d(\pi_{\theta}(S),\mathbb{P}(S)).
\]
\end{defn}
The definition encapsulates the special case where there exists a
parameter $\theta^{{\rm true}}$ such that $\mathbb{P}\equiv\pi_{\theta^{{\rm true}}}$.
We write the definition with a $\in$ sign to emphasize the fact that
this parameter may not be unique if the statistics are chosen poorly.
 This is not a standard definition 
for oracle paramaters. Typical parametric risk in Bayesian statistics have been measured
using the Kullback-Leibler divergence, Hellinger distance, Wasserstein distance etc.
In ABC without additional assumption on the statistics it is unrealistic to hope 
for such results. However it is interesting to note that for a finite collection of 
functions $\mathcal{F}=\left\lbrace S^1,S^2,\cdots,S^d\right\rbrace$ and using a distance based
on the $\sup$-norm the definition of the oracle risk $\inf_{\theta\in\Theta}\sup_{S\in\mathcal{F}}\vert\pi_\theta S-\mathbb{P}S\vert$ 
is similar in nature to a Wasserstein distances. Notice also that
$\delta(\theta,\theta^\prime S)=d(\pi_\theta S,\pi_{\theta^\prime})$ defines a pseudo metric.
We are interested in showing that the pseudo-posterior concentrates
in some sense around some $\theta^{\star}$.
We give some of the assumptions needed to do this in the next subsection.

\subsection{Assumptions and discussion\label{subsec:Assumptions-and-discussion}}

In order to say something on the convergence of this method, we must
set a few assumptions. In particular we want to state a result in
terms of the distance between moments of the summary statistics. To
get finite sample bounds on this quantity we need a concentration
inequality to hold under the probability $\mathbb{P}$. Our main assumption
on the data generating process is given by an Hoeffding type inequality.

\medskip{}

\begin{hyp}

\label{assu:Hoeffding}We say that the Hoeffding assumption is satisfied
for the model $m_{X}$, and a set $\mathcal{I}$ if for any $\lambda\in\mathcal{I}$
and some function $f:\mathbb{N}\times\mathcal{I}\mapsto\mathbb{R}_{+}$
we have 

\[
\left.\begin{array}{c}
m_{X}\otimes\mathbb{P}\left\{ e^{-\lambda\left(D_{n}^{S}-\mathbb{P}D_{n}^{S}\right)}\right\} \\
m_{X}\otimes\mathbb{P}\left\{ e^{-\lambda\left(\mathbb{P}D_{n}^{S}-D_{n}^{S}\right)}\right\} 
\end{array}\right\} \leq e^{f(n,\lambda)}.
\]

\end{hyp}

\medskip{}

The inequality can be interpreted as an integrated version (with respect
to $m_{X}$) of Hoeffding\textquoteright s inequality. If the distance
is bounded uniformly over $X^{n}$; then Hoeffding\textquoteright s
inequality will directly imply Assumption A\ref{assu:Hoeffding}.
Another case for which we can obtain such inequality is the bounded
difference inequality, sometimes referred to as McDiarmid's inequality
(\cite{Boucheron2013}). The expectation with respect to $m_{X}$
could, in specific cases, allow us to treat cases with unboundness.
Although more generally some authors have considered the case of unbounded
losses in a general context (\cite{grunwald2016,mendelson2014}).
Notice that the boundness condition does not have to come from the
data itself but can be due to the statistic $S$ or the distance itself.
Although interesting we leave the unbounded case for future studies.
The second assumption we will make is that the distance in jointly
convex in both of its argument.\medskip{}

\begin{hyp}

\label{assu:distance}

The distance $d(.,.):(x,y)\rightarrow d(x,y)$ is assumed to be convex
in both its arguments.

\end{hyp}

\medskip{}

This assumption is rather weak and is verified in particular for any
distance based on a norm. Assumption A\ref{assu:distance} is needed
to make the results more readable. In fact it can be removed if one
is satisfied with randomized estimators (see Remark \ref{rem:convex}
below).

Finally we introduce some regularity on the prior model, that is,
on the measurable map $\theta\mapsto\pi_{\theta}$. We require that
the moments of statisics with respect to the model are locally Lipschitz
around $\theta^\star$.

\medskip{}

\begin{hyp}

\label{assu:Lipschitz}We say that the model is $(\bar{\delta},L)$-locally
Lipschitz in $\theta^{\star}$ if there exists a $\bar{\delta}>0$
such that for any $\theta\in\left\{ \theta\in\Theta:\left\Vert \theta-\theta^{\star}\right\Vert <\delta\right\} $,
the moments of the statistics are $L$-Lipschitz. That is there exists $L<\infty$
such that 
\[
\Vert\pi_{\theta}\Vert_{Lip}=\sup_{\theta\in\left\{ \theta\in\Theta:\left\Vert \theta-\theta^{\star}\right\Vert <\delta\right\}}\frac{d\left(\pi_\theta S,\pi_{\theta^\star}S\right)}{\Vert\theta-\theta^\star\Vert}\leq L.
\]

\end{hyp}

\medskip{}

This kind of assumption ican rarely be checked in practice
as it depends on a model that usually does not have a closed form
(in interesting examples at least). 
\begin{rem}
A typical example where it is possible to check it is given by the
Gibbs model mentioned in the introduction where $\pi_{\theta}\left(dX^{n}\right)=\frac{1}{Z_{\theta}}e^{\theta^{T}S(X^{n})}\mu\left(dX^{n}\right)$
with bounded statistics $S$. Often those models do not have a tractable
normalizing constant and we might want to use ABC (see \cite{grelaud2009abc}). 
In this case we have that $\partial_\theta\log Z_\theta=\pi_\theta S$ and $\partial^2_{\theta\theta}\log Z_\theta=\pi_\theta \left\lbrace S-\pi_\theta S\right\rbrace^2$. Assumption (A\ref{assu:Lipschitz}) is satisfied if $\pi_\theta \left\lbrace S-\pi_\theta S\right\rbrace^2\leq L$ around $\theta^\star$.
\end{rem}
\begin{rem}
We can allow $L$ to depend on $n$, in the bounds that we prove $L$
will be compared to a sequence $\delta_{n}<\bar{\delta_{n}}$ that
will typically be an order of magnitude smaller than the rate of convergence.
This will allow some slack in the choice of the model.
\end{rem}
We are now ready to give the first building block for our inequalities. 

\section{Theoretical bounds\label{sec:Theoretical-bounds}}

In this section we give a few intermediate results to derive finite
sample oracle inequalities for the marginal distribution $\rho_{X,\lambda}$.
We will also give empirical bounds (Section (\ref{subsec:Empirical-bounds}))
and an adaptive version of the oracle inequality, that is a version
of the statistical estimator that automatically chooses the value
of the inverse temperature $\lambda$.

Our goal is to give bounds on the distance between expected values
of the summary statistics. This is a key quantity in our study. It appears to 
be the quantity that is minimised naturally when performing ABC with a window that
shrinks to $0$. Also $\delta(\theta,\theta^\star)=d(\pi_\theta S,\pi_{\theta^\star})$ 
is a pseudo-metrics that encompasses the flaws in the choice of the statistics.
We will state the inequalities in high probability
with respect to the data generating mechanism. We do not in particular
assume that the proposed likelihood matches the probability $\mathbb{P}$,
more formally we encapsulate the case where $\mathbb{P}\notin\left\{ \pi_{\theta},\theta\in\Theta\right\} .$ 
\begin{lem}
\label{lem:bound}Suppose that Assumptions A\ref{assu:Hoeffding},
A\ref{assu:distance} and A\ref{assu:Lipschitz} with constants $(\bar{\delta},L)$
are satisfied then, for any $\lambda\in\mathcal{I}$, $\delta<\bar{\delta}$
and $\epsilon>0$, with probability at least $1-\epsilon$, 

\begin{multline*}
d(\varrho_{\lambda,X}(S),\mathbb{P}(S))\leq\inf_{\theta\in\Theta}d(\pi_{\theta}(S),\mathbb{P}(S))+\sup_{\theta:\Vert\theta-\theta^\star\Vert<\delta}\pi_{\theta}\left\{ d\left(S,\pi_{\theta}S\right)\right\} +L\delta+\mathbb{P}\left\{ d\left(S,\mathbb{P}S\right)\right\} \\
+\frac{2f(n,\lambda)}{\lambda}-\frac{2}{\lambda}\log\pi\left(\left\{ \Vert\theta-\theta^{\star}\Vert<\delta\right\} \right)+\frac{2}{\lambda}\log\frac{2}{\epsilon}.
\end{multline*}
\end{lem}
\begin{proof}
The proof is given in Section (\ref{subsec:Proof-of-lemma1})
\end{proof}
The lemma states that the distance between the expectation of the
statistics under the pseudo-posterior and under the data generating
probability should get closer to each other, as the number of samples
increases, if the terms appearing on the right hand-side go to zero.
The first term is the bias due to the model misspecification, obviously
in the case where $\mathbb{P}\in\left\{ \pi_{\theta},\theta\in\Theta\right\} $
this term will be null. We can hope to make it small if the model
is rich enough. There exists a trade off : one could choose a statistics
that does not vary much to cancel the term at the cost of rendering
the left hand side insignificant. This will be detailed in the next
section when we give bounds on the parameters.

The two next terms $\sup_{\theta:\Vert\theta-\theta^\star\Vert<\delta}\pi_{\theta}\left\{ d\left(S,\pi_{\theta}S\right)\right\}$
and $\mathbb{P}\left\{ d(S,\mathbb{P}S)\right\} $ account for the
fact that the statistics should converge  to their expectation uniformily around $\theta^\star$ and under $\mathbb{P}$ 
(the convergence is in the metric $d$). This is rather weak, considering that we already
asked for the distance to have an exponential concentration. Most
of the time the statistics considered will be some empirical moment.
The terms $L\delta$ and $2f(n,\lambda)/\lambda$ are
defined in assumptions A\ref{assu:Hoeffding} and A\ref{assu:Lipschitz}
and account for the smoothness of the model and the rate at which
the distance concentrates respectively.

The last term to mention is the one depending on the prior. That is
the probability of a ball around the oracle parameter under $\pi$.
It is very common in the theoretical analysis of Bayesian estimators
and imposes that the prior puts sufficient mass around the optimal
parameter $\theta^{\star}$. We need to find a pair of converging
sequences $(v_{n},\delta_{n})$ such that for some constant $C>0$
we have $\log\pi\left(\left\Vert \theta-\theta^{\star}\right\Vert <\delta_{n}\right)\geq Cv_{n}$.
We will give examples of models for which those terms can be expected
to give the correct rate in Section (\ref{sec:Examples-of-bounds}).

Our goal needs also to be addressed. The lemma implies that we could
bound the distance between moments of statistics. First,
in all generality the statistics depend on the sample size and hence
the lemma might loose its interpretability. Second, one might wonder
as to why this statement is meaningful when the practitioner is interested
in the convergence of the estimator obtained as a moment of the distribution
$\rho_{\lambda}$ on $\Theta$ of Definition \ref{def:abc-theta}.
The key idea is to understand that the moment of the statistics are
supposed to identify or partly identify the distribution of interest.
In the case of methods of moments (in the correctly specified setting)
it is common to define a map $e_{h}:\theta\rightarrow e_{h}(\theta)\equiv\pi_{\theta}(h)$,
and to suppose it injective. In the next section we exploit further
this interpretation and give bounds on the parameter themselves. 
\begin{rem}
\label{rem:convex}In the rest of the paper we will use the convexity
assumption on the distance to ensure that $d(\varrho_{\lambda}S,\mathbb{P}S)\leq\varrho_{\lambda}\mathbb{P}d(S,S^{\prime})$.
In fact one could also state the results for the randomized estimator,
hence in deviation under the probability $\mathbb{P}\varrho_{\lambda}$,
thus removing the convexity assumption. We call randomized estimator
an estimator $\hat{\theta}_{n}$ that consists in a sample from $\varrho_{\lambda}$.
All the bounds actually hold on $\varrho_{\lambda}\mathbb{P}d(S,S^{\prime})$.
We will use this fact repeatedly in the rest of the paper.
\end{rem}

\subsection{Interpretation of the results}

In the previous subsection we have discussed a general lemma that
gives a bound on the excess risk of the distance between expected
moments of the statistics. Here we show that those results are amenable
under stronger assumptions to bounds in the parameter space.

\medskip{}

\begin{hyp}

\label{assu:inject}Assume that for $\theta^{\star}$ the oracle parameter
and $\forall\theta\in\Theta$ there exists $K<\infty$ such that 
\[
\left\Vert \theta-\theta^{\star}\right\Vert \leq Kd\left(\pi_{\theta}(S),\pi_{\theta^{\star}}(S)\right).
\]

\end{hyp}

\medskip{}

The assumption imposes a form of identifiability of the parameter
space given the statistics. For exponential models the assumption
has a more understandable formulation. Suppose that $\pi_{\theta}(dx)\propto\exp\left\{ -\theta^{T}S(x)\right\} dx;$
such models do not have a tractable normalizing constant and we might
want to use ABC (see \cite{grelaud2009abc}). The normalizing constant
$Z_{\theta}=\int\exp\left(-\theta^{T}S(x)\right)dx$ and hence under
appropriate regularity conditions $\left.\frac{\partial}{\partial\theta}\log Z_{\theta}\right|_{\theta}=\pi_{\theta}S$,
and $\left.\frac{\partial^{2}}{\partial\theta^{T}\partial\theta}\log Z_{\theta}\right|_{\theta}=\pi_{\theta}\left\{ \left(S-\pi_{\theta}S\right)^{2}\right\} $
hence if we assume $\inf_{\theta\in\Theta}\pi_{\theta}\left\{ \left(S-\pi_{\theta}S\right)^{2}\right\} >c>0$
then Assumption \ref{assu:inject} is satisfied by the mean value
theorem. 

In \cite{frazier2016} a weakened version of the assumption appears,
where the authors assume $\left\Vert \theta-\theta^{\star}\right\Vert \leq Kd^{\alpha}\left(\pi_{\theta}(S),\pi_{\theta^{\star}}(S)\right)$
for some constant $\alpha>0$.
\begin{rem}
\label{rem:K-constant}Notice that if $d$ is homogeneous and $K$
is known in advance, we can scale the statistic $S$ in such a way
that $K=1$. This is an important point as we will see later a constant
$K$ different from $1$ impact negatively the bias.

We need the following assumption on the exponential concentration
under the prior model 
\end{rem}
\begin{hyp}

\label{assu:Hoeffding-model}We say that the Hoeffding assumption
is satisfied for the model $m_{X}$, and a set $\mathcal{I}$ if for
any $\lambda\in\mathcal{I}$ and some function $\tilde{f}:\mathbb{N}\times\mathcal{I}\mapsto\mathbb{R}_{+}$
we have 

\[
m_{X}\left\{ e^{-\lambda\left(d(S,\pi_{\theta}S)-\pi_{\theta}d(S,\pi_{\theta}S)\right)}\right\} \leq e^{\tilde{f}(n,\lambda)}.
\]
\end{hyp}

\medskip{}

This is the same assumption as A\ref{assu:Hoeffding} only it is specified
under our prior model. If the distance is bounded or if we can use
McDiarmid's inequality in the original Hoeffding inequality, then
A\ref{assu:Hoeffding-model} will follow under independence or weak
dependence of the model.

We may now write the oracle inequality for the parameter $\theta$
sampled under the ABC pseudo posterior.
\begin{thm}
\label{thm:theta} Let assumptions A\ref{assu:Hoeffding}-\ref{assu:Hoeffding-model}
be satisfied, $\lambda\in\mathcal{I}$ and $\delta<\bar{\delta},$
then for any $n\geq1$ with probability $1-\epsilon$

\begin{multline*}
\varrho_{\lambda}\left\{ \left\Vert \theta-\theta^{\star}\right\Vert \right\} \leq3K\inf_{\theta\in\Theta}d(\pi_{\theta}(S),\mathbb{P}(S))+2K\sup_{\theta:\Vert\theta-\theta^\star\Vert<\delta}\pi_{\theta^{\star}}\left\{ d\left(S,\pi_{\theta^{\star}}S\right)\right\}+2KL\delta +2K\mathbb{P}\left\{ d\left(S,\mathbb{P}S\right)\right\} +K\sup_{\pi}\pi_{\theta}\left(d(\pi_{\theta}S,S)\right)\\
+\frac{2Kf(n,\lambda)}{\lambda}+\frac{K\tilde{f}(n,\lambda)}{\lambda}-\frac{4K}{\lambda}\log\pi\left(\left\{ \Vert\theta-\theta^{\star}\Vert<\delta\right\} \right)+\frac{K}{\lambda}\log\frac{32}{\epsilon^{3}}.
\end{multline*}
\end{thm}
\begin{proof}
The proof is given in Section \ref{subsec:thm-theta}.
\end{proof}
On top of the additional assumption that we impose on the model (exponential
concentration) we see that another requirement is that we have uniform
convergence in the sense of $\sup_{\pi}\pi_{\theta}\left(d(\pi_{\theta}S,S)\right)$
converging. Furthermore we notice that we do not recover an exact
oracle inequality unless $K\leq1/3$ (see Remark \ref{rem:K-constant}).
Theorem \ref{thm:theta} is particularly interesting when $\inf_{\theta\in\Theta}d(\pi_{\theta}(S),\mathbb{P}(S))$
is small. The model should be rich enough to have to be close to $\mathbb{P}$ for those moments at least. 
It is obvious that $d(\pi_{\theta}S,\mathbb{P}S)$ can be big, in fact take $\mathbb{P}$ to be a iid gausian $\mathcal{N}(0,\sigma^{2})$, 
chose the model $\pi_{\theta}=\delta_{\theta}$ the Dirac in $\theta$ and $S(X^{n})=Var(X^{n})$ we can see that as the variance $\sigma^{2}$ goes to $\infty$ so does $\inf_\theta d(\pi_{\theta}S,\mathbb{P}S)$. 
Of course the example points to a bad choice of statistics and model.

\subsection{Empirical bounds \label{subsec:Empirical-bounds}}

As a by-product of the oracle inequalities introduced in the previous
section we obtain empirical bounds. They offers a guarantee on the
generalization error of the algorithm. The upper bound can be computed
from the data, however they do not offer a convergence results in
the form of an oracle inequality. It will also come as one of the
building brick of our adaptive algorithm of sub-Section \ref{subsec:Adaptive-bounds}.

The following result is a rewriting of a result by \cite{Catoni2007}
adapted to our framework. We give a proof in the Section \ref{sec:Proofs-and-supporting}
for completeness.
\begin{prop}
\label{prop:Emp}Under assumptions A\ref{assu:Hoeffding} and A\ref{assu:distance}
for any $\rho\in\mathcal{M}_{1}^{+}(\Theta\times\mathcal{X}^{n})$,
$\lambda\in\mathcal{I}$ and $\epsilon>0$ with probability at least
$1-\epsilon$,
\[
d(\varrho(S),\mathbb{P}(S))\leq\varrho(d(S,S(Y^{n})))+\frac{1}{\lambda}\mathcal{K}(\varrho,\pi_{\theta}\pi)+\frac{f(n,\lambda)}{\lambda}+\frac{1}{\lambda}\log\frac{1}{\epsilon},
\]

As a special case for $\varrho_{\lambda}$ we have with probability at
least $1-\epsilon$,
\[
d(\varrho_{\lambda,X}(S),\mathbb{P}(S))\leq-\frac{1}{\lambda}\log Z_{\lambda,\pi}+\frac{f(n,\lambda)}{\lambda}+\frac{1}{\lambda}\log\frac{1}{\epsilon}.
\]
\end{prop}
\begin{proof}
The proof is given in Section \ref{prop:Emp}.
\end{proof}
Although the bound of Proposition \ref{prop:Emp} is computable from
the data it depends on the intractable integral $Z_{\lambda,\pi}$.
Given a probability level and inverse temperature we will give an
algorithm to compute an estimator of the upper-bound in sub-Section
\ref{subsec:Adaptive-bounds}. The reader familiar with SMC algorithms
will already recognize that we can compute the normalization constant
at almost no extra cost using this methodology. The empirical bound
will be used as a sanity check: it will allow us to get a bound on
the distance between moments showing that the pseudo-posterior actually
learns something from the data. For an example see Figure \ref{fig:emp-bound}.
The empirical bound gives a high probability upper bound of $d(\varrho_{\lambda,X}(S),\mathbb{P}(S))$,
if it is small it offers a garantee that $\inf_{\theta\in\Theta} d(\pi_\theta S,\mathbb{P}S)$ is
small too.

\subsection{Adaptive bounds\label{subsec:Adaptive-bounds}}

We have seen that the inverse temperature parameter is similar in
nature to the window of a Kernel-ABC algorithm. The main issue
with this, as for ABC, is the choice of the hyper-parameter. Some
authors (e.g. \cite{Ratman2009}) have considered putting a prior on the window in ABC. We use
a similar strategy and build a joint empirical bound on the sequence
of$\lambda$. The bound can then be minimized such as to obtain theoretical
guaranties on the adaptive algorithm. We therefore will define a prior
distribution $\nu$ on $(\mathcal{I},\mathcal{B}(\mathcal{I}))$.
The idea of using a prior to build a joint empirical bound that could
be minimized dates back to \cite{Catoni2007}. Although this strategy
is interesting in the sense that it reduces the dependence on the
hyper-parameter, it does make the assumption that the model is mis-specified.
In this sense it does allow one to recover faster rates in the case
where the model is actually correct. For strategies in the PAC-Bayesian
literature to adapt the parameter in such a way we refer the reader
to \cite{grunwald2011safe} (in the online scenario).

We start by defining the pseudo-posterior,
\begin{defn}
\label{def:adap-ABC}We call an adaptive-ABC (AdABC) pseudo posterior
at level $\epsilon$ the joint distribution
\[
\varrho_{\epsilon,\hat{\xi}(\lambda)}^{abc}(dX,d\theta)\propto e^{-\hat{\xi}(\lambda)d(S(X^{n}),S(Y^{n}))}\pi_{\theta}(dX^{n})\pi(d\theta),
\]

where the measure $\hat{\xi}$ on $(\mathcal{I},\mathcal{B}(\mathcal{I}))$
is defined as the minimizer of the empirical bound

\[
\inf_{\xi\in\mathcal{F}\subset\mathcal{M}_{+}^{1}(\mathcal{I})}\left\{ -\frac{1}{\xi(\lambda)}\log Z_{\epsilon,\xi}+\frac{1}{\xi(\lambda)}\left(\mathcal{K}(\xi,\nu)+\xi(f(n,\lambda))+\log\frac{1}{\epsilon}\right)\right\} 
\]

and $\mathcal{F}$ is a subset of the space of all probability measures
on $\mathcal{I}$.
\end{defn}
We obtain an aggregating measure $\hat{\xi}$ by minimizing a modification
of the empirical bound of proposition \ref{prop:Emp}. The minimization
is taken over a class of probability measures, of course the best
achievable bound would be the one minimizing $\xi$ over $\mathcal{M}_{+}^{1}(\mathcal{I})$,
this is however intractable. Instead we propose to use a variational
approximation, that is to replace the set of all probability measures
by a smaller but tractable subset. This is similar to the approach
taken by \cite{Alquier2015}, where the authors replace an exponential
weight aggregation with a variational approximation. There is obviously
a trade-off in the size of the family of approximation $\mathcal{F}$.
On one hand, we want it to remain small enough to be tractable, on
the other we want it to be large enough to get good theoretical properties. 

The algorithm depends on two quantities that need to be discussed.
First the function $f$ of Assumption A\ref{assu:Hoeffding} appears
in the variational approximation. Hence it will depend on the kind
of hypothesis that we have put on the probability $\mathbb{P}$. Second
we need to choose a level $\epsilon$ at which we want are theoretical
bounds to be true. In practice we observe that this level does not
have to much impact on the end estimator.

Using this estimator we can obtain a bound similar to that of Lemma
\ref{lem:bound},
\begin{lem}
\label{lem:oracle-adap}Suppose that assumptions A\ref{assu:Hoeffding},
A\ref{assu:distance} and A\ref{assu:Lipschitz} with constants $(\bar{\delta},L)$
are satisfied then for any $\delta<\bar{\delta}$ and $\epsilon>0$
we have with probability at least $1-\epsilon$,

\begin{multline*}
d(\varrho_{\epsilon,\hat{\xi}}(S),\mathbb{P}S)\leq\inf_{\theta\in\Theta}d(\pi_{\theta}(S),\mathbb{P}(S))+\sup_{\theta:\Vert \theta-\theta^\star\Vert}\pi_{\theta}\left\{ d(S,\pi_{\theta}S)\right\} +\delta L+\mathbb{P}\left\{ d(S,\mathbb{P}S)\right\} +\\
\inf_{\xi\in\mathcal{F}}\left[\frac{1}{\xi(\lambda)}\left\{ 2\xi\left[f(n,\lambda)\right]-2\log\pi\left(\left\{ \Vert\theta-\theta^{\star}\Vert<\delta\right\} \right)+2\mathcal{K}(\xi,\nu)+2\log\frac{2}{\epsilon}\right\} \right]
\end{multline*}
\end{lem}
As for Lemma \ref{lem:bound} we recover the main elements that we
expect to find in such a result. The difference is that the upper
bound is now explicitly optimized over $\xi$. To obtain a result
on specific example we need to find a class of probability for which
we can find an explicit solution of this problem. The easiest case
will be a parametric family $\mathcal{F}_{b}$ indexed by some parameter
$b$. We need to ensure that the following terms $\xi(\lambda),$
$\xi\left[f(n,\lambda)\right]$ and $\mathcal{K}(\xi,\nu)$ are tractable
and ensure a convergence result at the correct speed. The next section
is devoted to two examples for which we can obtain rates of convergence. 

\section{Examples of bounds\label{sec:Examples-of-bounds}}

In the next two sub-sections we show how the above bounds can be applied
to specific models. We give convergence rates associated to some specific
priors on the distance between expected statistics. 

\subsection{Bounds for the $\ell_{p}$ norm }

In this section we will use a specific distance, one based on the
$p$-norm hence we define,
\[
d\left(S(X^{n}),S(Y^{n})\right):=\Vert S(X^{n})-S(Y^{n})\Vert_{p},
\]
where $\Vert.\Vert_{p}$ is the $p$-norm. The summary statistics
are assumed to be empirical moments of a collection of $m$ bounded
functions $H=\left(h_{1},\cdots,h_{m}\right)$. Therefore we put $S(Y^{n})=\frac{1}{n}\sum_{i=1}^{n}H(Y_{i})=\left(\frac{1}{n}\sum_{i=1}^{n}h_{1}(Y_{i}),\cdots,\frac{1}{n}\sum_{i=1}^{n}h_{m}(Y_{i})\right).$
The boundness assumption is embodied by the existence of a constant
$K<\infty$ such that for any $x\in\mathcal{X}$, $\sup_{i\in\left\{ 1,\cdots,m\right\} }\left|h_{i}(x)\right|<K.$
We need to ensure that under those conditions Assumption A\ref{assu:Hoeffding}
is satisfied.
\begin{lem}
\label{lem:McDiarmid1}Let $(Y_{i})_{i=1}^{n}$ be sampled according
to $\mathbb{P}$ and the different $S$ and $d$ be given as above
then we have the following result,

\[
\left.\begin{array}{c}
\mathbb{P}\left\{ e^{-\lambda\left(D_{n}^{S}-\mathbb{P}D_{n}^{S}\right)}\right\} \\
\mathbb{P}\left\{ e^{-\lambda\left(\mathbb{P}D_{n}^{S}-D_{n}^{S}\right)}\right\} 
\end{array}\right\} \leq\exp\left\{ \frac{\lambda^{2}K^{2}m^{\frac{2}{p}}}{n}\right\} .
\]
\end{lem}
We have already noted that Assumption A\ref{assu:distance} is trivially
satisfied in the case of distances based on norms. Therefore we can
apply Lemma \ref{lem:bound} to any model that verifies the $(\bar{\delta},L,\alpha)$-Holder
Assumption A\ref{assu:Lipschitz}, with $f(n,\lambda)=2\frac{\lambda^{2}K^{2}m^{\frac{2}{p}}}{n}$
.
\begin{thm}
\label{thm:euclidean} Suppose that the prior model $\pi_{\theta}$
satisfies Assumption A\ref{assu:Lipschitz} with $(\bar{\delta},L)$,
that the model is independent $\pi_{\theta}(X^{n})dX^{n}=\prod_{i=1}^{n}\pi_{\theta}(X_{i})dX^{n}$,
and put 
\[
C:=\sqrt{K(p,m)\left[\max_{j\leq m}\mathbb{P}\left\{ \left(h_{j}-\mathbb{P}h_{j}\right)^{2}\right\} \vee\max_{j\leq m}\mathbb{\pi_{\theta^{\star}}}\left\{ \left(h_{j}-\mathbb{\pi_{\theta^{\star}}}h_{j}\right)^{2}\right\} \right]},
\]
 where $K(p,m):=\min\left(m,2e\log m,p-1\right)$, then we have for
any $\delta<\bar{\delta}$ with probability at least $1-\epsilon$,
\begin{multline*}
\Vert\varrho_{\lambda}^{ABC}(H)-\mathbb{P}(H)\Vert_{p}\leq\inf_{\theta\in\Theta}\Vert\pi_{\theta}(H)-\mathbb{P}(H)\Vert_{p}+\frac{2Cm^{\frac{1}{p}+1}}{\sqrt{n}}+L\delta+\frac{2\lambda K^{2}m^{\frac{2}{p}}}{n}-\frac{2}{\lambda}\log\pi\left(\left\{ \Vert\theta-\theta^{\star}\Vert<\delta\right\} \right)+\frac{2}{\lambda}\log\frac{2}{\epsilon}.
\end{multline*}
\end{thm}
\begin{rem}
The dependence in $m$ is not optimal, by doing the calculation in
the specific case of $p=2$ one would get a dependence in $m^{\frac{1}{2}}$
rather than $m^{\frac{3}{2}}$. 
\end{rem}
\begin{rem}
Under the assumptions on $h$ constant $C$  and that $p-1<m$ the constant $C$ is independent of the dimension of the problem.
\end{rem}

It remains to choose a prior putting sufficient mass around $\theta^{\star},$
and to optimize the bound in $\delta$, $\lambda$ and the hyper-parmeters
of prior. For ease of exposition let us take a model with parameter
space $\Theta=\mathbb{R}^{d}$ and a Gaussian prior with mean $0$
and isotropic variance $\vartheta I_{p}.$ In this case we can get
the following corollary,
\begin{cor}
\label{cor:euclidean-Gaussian} Assume the hypotheses of Theorem \ref{thm:euclidean},
in addition suppose that $\Theta=\mathbb{R}^{d}$ with the Euclidean
norm and that $\pi(d\theta)=\Phi_{d}(d\theta;0_{d},\vartheta I_{d})$
with $\vartheta>0$, and $\left\Vert \theta^{\star}\right\Vert _{2}\leq1$.
Put $\lambda=\sqrt{\frac{dn}{K^{2}m^{\frac{2}{p}}}}$ and $\delta=\sqrt{\frac{\vartheta}{n}}$
then with probability at least $1-\epsilon$,

\begin{multline*}
\Vert\varrho_{\lambda}^{ABC}(H)-\mathbb{P}(H)\Vert_{p}\leq\inf_{\theta\in\Theta}\Vert\pi_{\theta}(H)-\mathbb{P}(H)\Vert_{p}+\frac{2Cm^{\frac{1}{p}+1}}{\sqrt{n}}+L\delta+2K\sqrt{\frac{d}{n}}m^{\frac{1}{p}}+\\
2K\sqrt{\frac{d}{n}}m^{\frac{1}{p}}\left\{ \frac{1}{2}\log\left(8\pi nd\right)+\frac{1}{\vartheta}+\frac{1}{nd}\right\} +2K\frac{m^{\frac{1}{p}}}{\sqrt{dn}}\log\frac{2}{\epsilon}.
\end{multline*}
\end{cor}
The choice of lambda is taken as to optimize the rate of the oracle inquality, the same rate can be achieved using the adaptive method presented in Theorem \ref{thm:Adap}.
Ignoring the log terms we get a rate of $\mathcal{O}\left(\sqrt{\frac{d}{n}}m^{\frac{2}{p}}\right)+\mathcal{O}\left(\sqrt{\frac{1}{n}}m^{\frac{1}{p}+1}\right).$
The dependence in the size of the statistic comes from the convergence
in the Euclidean norm of the statistics themselves. The dependence
in $n$ and $d$ is expected although we do not have theoretical result
proving optimality in this context.

It remains to prove that this result can also be achieved using AdABC.
As for the above theorem, based on Lemma \ref{lem:McDiarmid1} we
can apply the result of Lemma \ref{lem:oracle-adap} to the framework
of this section. 
\begin{thm}
\label{thm:Adap}Suppose that the model $\pi_{\theta}$ satisfies
Assumption A\ref{assu:Lipschitz}, that the model is independent $\pi_{\theta}(X^{n})dX^{n}=\prod_{i=1}^{n}\pi_{\theta}(X_{i})dX^{n}$,
and put 
\[
C:=\sqrt{K(p,m)\left[\max_{j\leq m}\mathbb{P}\left\{ \left(h_{j}-\mathbb{P}h_{j}\right)^{2}\right\} \vee\max_{j\leq m}\mathbb{\pi_{\theta^{\star}}}\left\{ \left(h_{j}-\mathbb{\pi_{\theta^{\star}}}h_{j}\right)^{2}\right\} \right]}
\]
 where $K(p,m):=\min\left(m,2e\log m,p-1\right)$, then for we have
with probability at least $1-\epsilon$, 
\begin{multline*}
\left\Vert \varrho_{\epsilon,\hat{\xi}}(H)-\mathbb{P}H\right\Vert _{p}\leq\inf_{\theta\in\Theta}\left\Vert \pi_{\theta}(H)-\mathbb{P}(H)\right\Vert _{p}+2\frac{Cm^{\frac{1}{p}+1}}{\sqrt{n}}+L\delta\\
+\inf_{\xi\in\mathcal{F}}\left[\frac{1}{\xi(\lambda)}\left\{ 2\xi\left[\lambda^{2}\right]\frac{K^{2}m^{\frac{2}{p}}}{n}-2\log\pi\left(\left\{ \Vert\theta-\theta^{\star}\Vert<\delta\right\} \right)+2\mathcal{K}(\xi,\nu)+2\log\frac{2}{\epsilon}\right\} \right]
\end{multline*}
\end{thm}
\begin{proof}
We start by applying Lemma \ref{lem:oracle-adap}, with probability
at least $1-\epsilon$,
\begin{multline*}
d(\varrho_{\epsilon,\hat{\xi}}(S),\mathbb{P}S)\leq\inf_{\theta\in\Theta}d(\pi_{\theta}(S),\mathbb{P}(S))+\pi_{\theta^{\star}}\left\{ d(S,\pi_{\theta^{\star}}S)\right\} e^{L\delta^{\alpha}}+\mathbb{P}\left\{ d(S,\mathbb{P}S)\right\} \\
+\inf_{\xi\in\mathcal{F}}\left[\frac{1}{\xi(\lambda)}\left\{ 2\xi\left[\frac{\lambda^{2}K^{2}m^{\frac{2}{p}}}{n}\right]-2\log\pi\left(\left\{ \Vert\theta-\theta^{\star}\Vert<\delta\right\} \right)+2\mathcal{K}(\xi,\nu)+2\log\frac{2}{\epsilon}\right\} \right].
\end{multline*}

As for Theorem \ref{thm:euclidean} we can easily bounds the terms
$\pi_{\theta^{\star}}\left\{ d(S,\pi_{\theta^{\star}}S)\right\} $
and $\mathbb{P}\left\{ d(S,\mathbb{P}S)\right\} $. This yields the
result.
\end{proof}
We can apply Theorem \ref{thm:Adap} to the case of Gaussian prior
on $\Theta=\mathbb{R}^{d}$ as before. We will need to add to things
to the previous result. First we need a prior on the inverse temperature.
Second we need a family of approximating distribution. As an example
set the prior to be an exponential distribution with parameter $\alpha$.
The family of approximation is taken also as an exponential distribution
$\mathcal{F}=\left\{ \xi(d\lambda)=\beta e^{-\beta\lambda}d\lambda,\beta\in\mathbb{R}_{+}\right\} $.
We have the following corollary,
\begin{cor}
\label{cor:adap-gaussian }Assume the hypotheses of theorem \ref{thm:euclidean},
in addition suppose that $\Theta=\mathbb{R}^{d}$ with the Euclidean
norm and that $\pi(d\theta)=\Phi_{d}(d\theta;0_{d},\vartheta I_{d})$
with $\vartheta>0$, and $\left\Vert \theta^{\star}\right\Vert _{2}\leq1$.
Let $\nu(d\lambda)=\alpha e^{-\alpha\lambda}$ and $\mathcal{F}=\left\{ \xi(d\lambda)=\beta e^{-\beta\lambda}d\lambda,\beta\in\mathbb{R}_{+}\right\} $
and put $\delta=\sqrt{\frac{d}{n}}<\bar{\delta}$ then with probability
at least $1-\epsilon$,
\begin{multline*}
\left\Vert \varrho_{\epsilon,\hat{\xi}}(H)-\mathbb{P}H\right\Vert _{p}\leq\inf_{\theta\in\Theta}\left\Vert \pi_{\theta}(S)-\mathbb{P}(S)\right\Vert _{p}+\frac{2Cm^{\frac{1}{p}}}{\sqrt{n}}e^{L\left(\frac{\vartheta}{n}\right)^{\alpha}}+4Km^{\frac{1}{p}}\sqrt{\frac{d}{n}}+2Km^{\frac{1}{p}}\sqrt{\frac{d}{n}}\left\{ \frac{1}{2}\log\left(8\pi dn\right)+\frac{1}{\vartheta}+\frac{1}{dn}\right\} \\
+\frac{Km^{\frac{1}{p}}}{\sqrt{nd}}\log\left(\frac{K^{2}m^{\frac{2}{p}}}{nd\alpha}\right)+2\frac{Km^{\frac{1}{p}}}{\sqrt{nd}}\log\frac{2}{\epsilon}.
\end{multline*}
\end{cor}
Again ignoring log terms we get the same rate as for the non adaptive
case. We have to impose that the parameter of the approximated posterior
over the inverse temperature is larger than the prior. In practice
this is not a problem as the Kullback Leibler terms is infinite otherwise.
\begin{rem}
The dependence of the bound on the dimension of the statistics suggests
a bias-variance trade-off in the size of the statistics. The latter
has an effect on how small the term $\inf_{\theta\in\Theta}\left\Vert \pi_{\theta}(S)-\mathbb{P}(S)\right\Vert _{p}$
can be made and the rate of convergence $\mathcal{O}\left(\sqrt{\frac{d}{n}}m^{\frac{2}{p}}\right)+\mathcal{O}\left(\sqrt{\frac{1}{n}}m^{\frac{1}{p}+1}\right).$ 
\end{rem}

\subsection{Nonparametric bound}

We propose an application of the bounds in the case where $\Theta=\mathcal{H}^{\beta}\left(\left[0,1\right]\right),$
the class of $\beta$- Holder function on $\left[0,1\right]$ (see p.6 \cite{tsybakov2003}). Suppose
that the data is generated by the following black box: generate an
iid $n$-sample $(\xi_{i})_{i=1}^{n}$ according to some distribution
$\mathbb{G}$ and observe $\left(Y_{i}\right)_{i}$ given by $Y_{i}=F_{i}(\xi_{i},f_{0}),$
were $f_{0}$ is some unknown function in $\mathcal{F}\neq\Theta$ and $F_{i}$ is a known application
in $\left[-K,K\right]\subset\mathbb{R}$. As before the probability
distribution of the sample is denoted $\mathbb{P}$. The summary statistic
is the whole sample and the distance between two samples of size $n$
is taken to be the empirical $L_{2}$ distance multiplied by $\frac{1}{\sqrt{n}}$
\[
d(X_{1}^{n},Y_{1}^{n})=\frac{1}{\sqrt{n}}\Vert X^{n}-Y^{n}\Vert_{2,n}=\sqrt{\frac{1}{n^{2}}\sum_{i=1}^{n}(X_{i}-Y_{i})^{2}}.
\]

The model that we have described allows us as for the previous section
to get a concentration inequality.
\begin{lem}
\label{lem:McDiarmid-nonparam}With the distance and the statistics
defined above, for any $\lambda\in\mathbb{R}_{+}$ we get 
\[
\left.\begin{array}{c}
\mathbb{P}\left\{ e^{\lambda\left(D_{n}^{S}-\mathbb{P}D_{n}^{S}\right)}\right\} \\
\mathbb{P}\left\{ e^{\lambda\left(\mathbb{P}D_{n}^{S}-D_{n}^{S}\right)}\right\} 
\end{array}\right\} \leq e^{\frac{\lambda^{2}K}{2n}}.
\]
\end{lem}
In this section we propose a non-parametric model. We put a Gaussian
process prior on $\Theta$. We assume furthermore that this prior
is centered and that the kernel is Gaussian, $k(s,t):=\exp\left(-\left|s-t\right|^{2}\right)$.
For more developments on this kind of prior see \cite{Rasmussen2006}.
In the course of the proof we will see that we are not limited to
Gaussian kernels, but in fact Gaussian processes with exponential
tail spectral measure \cite{van2007bayesian}.
\begin{defn}
Given a Gaussian process $W=\left(W_{t}:t\in[0,1]\right)$ we define
the re-scaled version for a scaling constant $c_{n}$ as 
\[
t\mapsto W_{t/c_{n}}.
\]
\end{defn}
We will use this re-scaled version to build a prior on $\Theta$.
We use the sup norm on $\mathcal{H}^{\beta}\left(\left[0,1\right]\right)$
we write in this section $\left\Vert f\right\Vert =\sup_{t\in[0,1]}\left|f(t)\right|.$
We will need to bound the small ball probability under the prior,
fortunately many results are known for this problem, we can use in
particular \cite{van2007bayesian}.
\begin{lem}
\label{lem:small-ball}Let $f^{\star}\in\mathcal{H}^{\beta}\left(\left[0,1\right]\right)$
then for any positive decreasing sequences $c_{n}$, $\delta_{n}$,
such that $\delta_{n}\geq c_{n}^{\beta}$ there exists constants $C_{0}$
and $D_{0}$ depending on $f^{\star}$ and a rank $n_{0}\in\mathbb{N}$
such that $\forall n\geq n_{0}$ we have 
\[
-\log\pi\left[\left\Vert f-f^{\star}\right\Vert <\delta_{n}\right]\leq D_{0}\left(\frac{1}{c_{n}}\right)+C_{0}\frac{1}{c_{n}}\left(\log\frac{1}{c_{n}\delta_{n}}\right)^{2}.
\]
\end{lem}
We use the above lemma to give a result on a non parametric estimator,
\begin{thm}
\label{thm:nonpara}Suppose that Assumption A\ref{assu:Lipschitz}
is verified for the prior model and the condition described above
on the prior and the model are verified, assume 
\[
C:=\sqrt{\sup_{i\leq m}\mathbb{P}\left\{ \left(X_{i}-\mathbb{P}X_{i}\right)^{2}\right\} \vee\sup_{i\leq m}\pi_{\theta^{\star}}\left\{ \left(X_{i}-\mathbb{\pi_{\theta^{\star}}}X_{i}\right)^{2}\right\} }<\infty
\]
 $c_{n}\asymp\left(\frac{\log^{2}n}{n}\right)^{\frac{1}{2\beta+1}}$
and $\lambda_{n}\asymp n^{\frac{\beta+1}{2\beta+1}}\left(\log n\right)^{\frac{\beta}{2\beta+1}}$
then there exists constant $K_{2},K_{1}>0$ such that for $n$ large
enough we get with probability at least $1-\epsilon,$

\begin{multline*}
\left\Vert \varrho_{\lambda}^{ABC}(X^{n})-\mathbb{P}(X^{n})\right\Vert _{2,n}\leq\inf_{\theta\in\Theta}\left\Vert \pi_{\theta}(X^{n})-\mathbb{P}(X^{n})\right\Vert _{2,n}+K_{2}n^{-\frac{\beta}{2\beta+1}}\left\{ \log n\right\} ^{\frac{\beta}{2\beta+1}}\\
+K_{1}n^{-\frac{\beta+1}{2\beta+1}}\left(\log n\right)^{\frac{\beta}{2\beta+1}}\log\frac{2}{\epsilon}.
\end{multline*}
\end{thm}
\begin{rem}
Although the bounds allow for convergence in terms of moments of statistics,
the choice of a statistic that would allow correct inference on the
function $f$ remains open. Note that for the max-norm the dependence
in the dimension of the statistic does not appear in the bound. This
leaves hope that one may be able to choose a rich enough class of
statistics. A full study of this problem is beyond the aim of this
paper. We simply use this idea as an illustration of how the bounds
can be used.
\end{rem}

\section{Monte Carlo algorithm\label{sec:Monte-Carlo-algorithm}}

The previous sections were devoted to the introduction of new theoretical
results. We have used a previously known definition of the pseudo-posterior
(Equation \ref{eq:Kernel-ABC}) applied to the specific case of the
the exponential kernel. By studying this case we get precise theoretical
results, robust to misspecification. In this section we build upon
known samplers to propose an efficient implementation of the statistical
procedure. 

The ABC pseudo-posterior can be expressed as a joint distribution
with latent variables $X^{n}\in\mathcal{X}^{n}$. We can use most
of the Monte Carlo toolbox available in this context. In this paper
we will build upon the adaptive algorithm of \cite{DelMoral2012}.
We want to sample from the pseudo-posterior of Definition \ref{def:abc-theta},
to do this the authors propose an adaptive sequential Monte Carlo
algorithm (SMC) on the joint distribution of Equation (\ref{eq:Kernel-ABC})
with uniform kernel, 
\[
\pi_{\epsilon}(d\theta,dX^{n})\propto\textbf{1}_{d\left(S(X^{n}),S(Y^{n})\right)\leq\epsilon}\pi_{\theta}(dX_{1}^{n})\pi(d\theta).
\]
 The uniform kernel is akin an accept-reject step, the weights take
value $0$ or $1$. 

The algorithm proposed in this paper builds on a few known results,
we shortly describe some of them in the next section. For further
insights the reader is referred to (\cite{DelMoral2012,DelMoral2006}). 

\subsection{Sequential Monte Carlo}

In this subsection we describe SMC as a general algorithm to sample
from a sequence of probability distribution $\left(\pi_{n}\right)_{n\in\mathbb{T}}$
for some index set $\mathbb{T}$ on the measurable space $\left(E,\mathcal{E}\right)$
as presented in \cite{DelMoral2006}.

Each distribution at index $n$ is approximated by a collection of
$M$ random variables $(Z_{i,n})_{i=1}^{M}$ termed particles. We
move the array of particles using the kernel $K_{n}(z_{n-1},dz_{n})$,
and use importance sampling to correct the distribution. Unfortunately
if the distribution of the particles are distributed according to
$\eta_{n-1}$ at index $n-1$ applying the kernel yields a distribution
$\eta_{n}(dx^{\prime})=\int_{E}\eta_{n-1}(dx)K_{n}(x,dx^{\prime})$,
which is not tractable in general. \cite{DelMoral2006} suggest to
write the importance sampling step on the extended space $\left(E^{n+1},\mathcal{E}^{\otimes n+1}\right)$
by introducing a backward kernel $L_{n}$. The importance sampling
algorithm is done between the joint distribution $\eta_{n}(x_{1}^{n})$
and $\tilde{\pi}_{n}(x_{1}^{n})=\pi_{n}(x_{n})\prod_{k=1}^{n-1}L_{k}(x_{k+1},dx_{k})$,
the vanilla algorithm is summed up in algorithm \ref{alg:SMC-vanilla}.

\begin{algorithm}
\caption{Vanilla SMC}
\label{alg:SMC-vanilla}

 \begin{description} 
\item[Input] $M$ (number of particles),  $\tau\in(0,1)$ (ESS threshold),
\item[Init.] Sample $Z_0^i\sim\pi_0$ for $i=1$ to $N$, put $W_{i,0}=\frac1M$
\item[Loop]
\begin{description} 
\item[a.] If $ESS(Z_{n-1}^i)=\frac{\{\sum_{i=1}^M w_n(Z_{n-1}^i)\}^2} {\sum_{i=1}^M \{w_n(Z_{n-1}^i))^2\} } \leq \tau M $, then Resample the particles i.e. draw $A_n^i$ in $1,\ldots,M$ so that
$$\mathbb{P}(A_n^i=j) = w_n(Z_{n-1}^j)/\sum_{k=1}^M w_n(Z_{n-1}^k);$$
see Algorithm \ref{alg:SystResampling} in the appendix.
\item[b.] Sample $Z_n^i \sim K_n(Z_{t-1}^{A_n^i},dz)$ for $i=1$ to $N$ and compute 
$$
W_{n}^i\propto W^i_{n-1}\frac{\pi_{n}(Z_n)}{\pi_{n-1}(Z_{n-1})}\frac{\rm{d}L_{n-1}(Z_n^i,Z_{n-1}^i)}{\rm{d}K_{n-1}(Z_{n-1}^i,Z_{n}^i)}
$$
\end{description} 
\end{description} 
\end{algorithm}

As in \cite{DelMoral2012}, we choose the kernel $K_{n}$ to be invariant
with respect to the distribution $\pi_{n}$, and $L_{n}$ as the reversal
backward kernel. In this case the weight update is given by 
\[
W_{n,i}\propto W_{n-1,i}\frac{\pi_{n}(Z_{n-1})}{\pi_{n-1}(Z_{n-1})}.
\]
To adapt the algorithm to the ABC framework we will take the sequence
of distribution as a sequence of pseudo-posterior distributions with
decreasing temperatures. Take the sequence of target distribution
$\left\{ \pi_{\epsilon_{n}}(d\theta,dX^{n})\right\} $ for a decreasing
sequence of $\epsilon_{n}$. In fact we can use in all generality, 
for any $M\in\mathbb{N_{\star}}$
\[
\pi_{\epsilon_{n}}^{M}(d\theta,dX_{1:M}^{n})\propto\left(\frac{1}{M}\sum_{i=1}^{M}\textbf{1}_{d\left(S(X^{n,i}),S(Y^{n})\right)\leq\epsilon_{n}}\right)\prod_{i=1}^{M}\pi_{\theta}(dX^{n,i})\pi(d\theta).
\]
 This sequence of distributions has the same marginals as the previous
ones. This allows us to rewrite the weight update as,
\[
W_{n,i}\propto W_{n-1,i}\frac{\frac{1}{M}\sum_{i=1}^{M}\textbf{1}_{d\left(S(X^{n,i}),S(Y^{n})\right)\leq\epsilon_{n}}}{\frac{1}{M}\sum_{i=1}^{M}\textbf{1}_{d\left(S(X^{n,i}),S(Y^{n})\right)\leq\epsilon_{n-1}}}.
\]
 In sub-Section \ref{sec:Monte-Carlo-algorithm} we will adapt the
algorithm to the case of exponential kernels and propose an adaptive
algorithm for the choice of parameter $M$ that allows to address
the pitfalls of the algorithm of \cite{DelMoral2012}.

\subsection{SMC-ABC for exponential kernels}

We have described above the main building blocks of SMC-ABC. Here
we give more insight on the algorithm and propose some methodological
improvements. Our goal is to sample from 
\begin{equation}
\varrho_{\lambda}(d\theta)=\frac{1}{Z_{\lambda,\pi}}\int_{\mathcal{X}}e^{-\lambda d(S(X^{n}),S(Y^{n}))}\pi_{\theta}(dX^{n})\pi(d\theta).\label{eq:posterior2}
\end{equation}
 As discussed in the previous section we are in fact going to sample
from the joint distribution 
\[
\pi_{\lambda}^{M}(d\theta,dX_{1:M}^{n})\propto\left(\frac{1}{M}\sum_{i=1}^{M}e^{-\lambda d\left(S(X^{n,i}),S(Y^{n})\right)}\right)\prod_{i=1}^{M}\pi_{\theta}(dX^{n,i})\pi(d\theta),
\]
hence using multiple draw from the prior model $\pi_{\theta}$. The
marginal in $\theta$ of the distribution thus defined is still given
by equation \ref{eq:posterior2}. It is straightforward to adapt the
SMC methodology to this particular example. We need to define an increasing
sequence of inverse temperature $0=\lambda_{0}<\lambda_{1}<\cdots<\lambda_{T}=\lambda$,
thus defining a sequence of distribution to sample from $\pi_{t}\equiv\pi_{\lambda_{t}}^{M}$.
Hence the weight update in algorithm \ref{alg:ABC-SMC} is given by
\[
W_{t,i}\propto W_{t-1,i}\frac{\sum_{i=1}^{M}e^{-\lambda_{t}d\left(S(X^{n,i}),S(Y^{n})\right)}}{\sum_{i=1}^{M}e^{-\lambda_{t-1}d\left(S(X^{n,i}),S(Y^{n})\right)}}.
\]
The algorithm depends on several inputs: first the index of the sequence
of distribution $\left(\lambda_{t}\right)_{t>0}.$ In \cite{DelMoral2012}
the authors choose the sequence of windows of the uniform kernel adaptively.
Similarly we propose to adapt it to a variance criterion on the weights.
At each step $t$ we therefore choose the next $\lambda_{t+1}$ according
to the past weighted particles. Second, we must choose the MCMC kernel.
Finally we also need to choose adaptively the number of samples $M$
drawn from the model. We give a description of what was mentioned
in a general pseudo code (Algorithm \ref{alg:ABC-SMC}).

\begin{algorithm}[H]
\caption{SMC-ABC}
\label{alg:ABC-SMC}

 \begin{description} 
\item[Input] $N$ (number of particles),  $\tau\in(0,1)$ (ESS threshold), optimal acceptance ratio $\alpha$.
\item[Init.] Sample $Z_0^i\sim\pi_0$ for $i=1$ to $N$, put $W_{i,0}=\frac1N$, set $M=1$,  
\item[Loop]
\begin{description} 
\item[a.] Choose $\lambda_t$ such that $ESS(\theta_t^i,X_t^i)=\tau N $ (Section \ref{sec:inv-temp}), 
\item[b.] Resample the particles i.e. draw $A_n^i$ in $1,\ldots,N$ so that
$$\mathbb{P}(A_n^i=j) = w_n(\theta_t^i,X_t^i)/\sum_{k=1}^M w_n(\theta_t^i,X_t^i);$$
see Algorithm \ref{alg:SystResampling} in the appendix.
\item[c.] Sample 
$(\theta_t^i,X_t^i)_n^i \sim K_n\left((\theta_{t-1}^{A_n^i},X_{t-1}^{A_n^i}),dz\right)$ where $K_n$ is a MCMC Kernel (see Section \ref{sec:MCMC}). 
\item[d.]for $i=1$ to $N$ and compute 
$$
W_{t,i}\propto W_{t-1,i}\frac{\sum_{i=1}^{M}e^{-\lambda_{t}d\left(S(X^{n,i}),S(Y^{n})\right)}}{\sum_{i=1}^{M}e^{-\lambda_{t-1}d\left(S(X^{n,i}),S(Y^{n})\right)}}
$$
\item[e.] Choose new value of $M$ according to some criterion (Section \ref{sec:Choice-M})
\end{description} 
\end{description} 
\end{algorithm}

The complexity of the algorithm will be dominated by the adaptation
in the $\theta$ dimension as we will need to estimate and invert
a covariance matrix. The complexity is linear in the number of particles.
Hence overall we get a complexity of $\mathcal{O}\left(d^{3}+MNnm\right)$
as long as the statistics are linear in the number of samples and
that the distance is linear in the number of statistics.

The memory cost will be dominated by the number of particles at any
given time, we need to keep all $\left\{ \theta_{i},S(X^{n})_{i}^{1:M}\right\} _{i=1}^{N}$
hence with the notation of the previous sections $\mathcal{O}\left(dN+mMN\right)$.
Keeping track of the current system of statistics, will usually be
by far the most expensive. In most cases of practical interest the
dimension of the statistics will be much smaller than $n$ and $M$
will typically be of the order of tens. 

In the following we propose several ways to write a more efficient
algorithm based on the current approximation of the likelihood. 

\subsection{Choice of the MCMC kernel}

\label{sec:MCMC}

We need to build a MCMC kernel with invariant distribution, 
\[
\rho(\theta,X^{n,M})\propto\left(\frac{1}{M}\sum_{i=1}^{M}e^{-\lambda d\left(S(X^{n,i}),S(Y^{n})\right)}\right)\prod_{i=1}^{M}\pi_{\theta}(dX^{n,i})\pi(d\theta).
\]
 We use a pseudo-marginal algorithm \cite{Andrieu2010}, i.e. a Random
walk Metropolis-Hastings (RW-MH) algorithm in the augmented state
space with a specific proposal. We let $\mathcal{Q}\left(\left(\theta,X^{n,M}\right)_{t-1},d\vartheta\times dY^{n,M}\right)$
be a kernel on $\left(\Theta\times\mathcal{X}^{n\times M}\right)$
such that 
\[
\mathcal{Q}\left(\left(\theta,X^{n,M}\right)_{t-1},d\vartheta\times dY^{n,M}\right)=\prod_{i=1}^{M}\pi_{\vartheta}(dY^{n,i})q(\theta,d\vartheta),
\]
 where $q(\theta,d\vartheta)$ is the kernel of a Gaussian random
walk centered in $\theta$. It well known that this proposal allows
one to build a MH algorithm on the joint space with acceptation ratio

\begin{equation}
\alpha\left(\left(\theta,X^{n,M}\right)_{s-1},\left(\theta,X^{n,M}\right)_{s}\right)=1\wedge\frac{\left(\sum_{i=1}^{M}e^{-\lambda d\left(S(X^{n,i})_{s},S(Y^{n})\right)}\right)\pi(\theta_{s})}{\left(\sum_{i=1}^{M}e^{-\lambda d\left(S(X^{n,i})_{s-1},S(Y^{n})\right)}\right)\pi(\theta_{s-1})}.\label{eq:MHratio}
\end{equation}
The algorithm itself can also be thought as a specific instance of
an ABC-MCMC kernel. One of the key advantage of using the kernel in
a SMC algorithm is that we can calibrate it using the particles at
a given time. The proposal, in $\theta$ space, is taken to be a Gaussian
distribution centered at the previous value with the estimated covariance
matrix multiplied by a scaling factor. In the case where the dimension
$d$ of the parameter space is large we perform a penalized estimation
of the covariance estimation to ensure it is positive definite. 

Furthermore we propose applying the kernel a fixed number of steps
$K$. Note that it is also possible to calibrate the number of steps
on the average distance the swarm of particles moves as in \cite{Ridgway2014}.
Here to dissociate the effects of the different changes in the methodology
we fix it to some amount ($K=3$ in all experiments). The algorithm
for the MCMC move is summarized in the following pseudo-code (Algorithm
\ref{alg:ABC-MCMC}).

\begin{algorithm}[H]
\caption{MCMC kernel}
\label{alg:ABC-MCMC}

 \begin{description} 
\item[Given] $\theta_{t-1}$ and $S(X^n)_{t-1}$ 
\begin{description}    
\item[a.] Sample 
$(\theta^{\prime},X^{\prime n})\sim \pi(d X^n \vert \theta)q(d\theta\vert\theta_{t-1})$     \item[b.] Compute the log acceptance ratio         
\[         
\ell_t=\log\frac{\sum_{i=1}^{M}e^{-\lambda d\left(S(X^{n,i})_{s},S(Y^{n})\right)}}{\sum_{i=1}^{M}e^{-\lambda d\left(S(X^{n,i})_{s-1},S(Y^{n})\right)}} +\log\frac{\pi(\theta^\prime)}{\pi(\theta_{t-1})} 
\]
\item[c.] Sample $U\sim \mathcal{U}_{[0,1]}$         
\begin{description}             
\item[If] $\ell_t<\log U$             
\item[Then] Set $(\theta_t,S(X^n))\leftarrow (\theta^\prime,S(X^{n,\prime}))$             \item[Else] Set $(\theta_t,S(X^n))\leftarrow (\theta_{t-1},S(X^{n})_{t-1})$         \end{description} 
\end{description} 
\end{description} 
\end{algorithm}

\subsection{Choice of the number of samples $M$}

\label{sec:Choice-M}

Increasing $M$ reduces the variance of the weights and the acceptance
rate. The weights are already controlled by the speed at which the
sequence of inverse temperature is increased. It is therefore reasonable
to increase the value of $M$ whenever the acceptance ratio goes below
some threshold. This is akin the idea that was used in \cite{Chopin2013a}
in the context where the samples are generated according to a particle
filter. The author double the number of samples whenever the acceptance
goes below $20\%$. In their paper the change in the number of samples
is done via importance sampling on an extended space. We detail this
strategy in sub-Section \ref{subsec:Importance-sampling} and propose
an approach adapted from \cite{Chopin2015} in sub-Section \ref{subsec:Gibbs-sampling}.

\subsubsection{Importance sampling\label{subsec:Importance-sampling}}

The first idea to change the number of samples $M$ is based on an
importance sampling algorithm on the extended space including the
new sample $\tilde{X}^{n,1:\tilde{M}}$ of size $\tilde{M}$. The importance sampling step is performed
on $\left(\theta,X^{n,1:M},\tilde{X}^{n,1:\tilde{M}}\right)$. Hence
we want to propose a set of $\tilde{M}$ new particles sampled from
$\prod_{i=1}^{\tilde{M}}\pi\left(\left.d\tilde{X}^{n,i}\right|\theta\right)$.
We start, on the extended space, with an approximate sample from the
distribution 

\[
q\left(d\theta,dX^{n,1:M},dX^{n,1:\tilde{M}}\right)=\rho_{\lambda}\left(d\theta,dX^{n,1:M}\right)\prod_{i=1}^{\tilde{M}}\pi\left(\left.d\tilde{X}^{n,i}\right|\theta\right).
\]
To perform importance sampling we need to introduce an auxiliary backward
kernel $L\left(\tilde{X}^{n,1:\tilde{M}},dX^{n,1:M}\right).$ Using
the proposal distribution $q$ we thus define the importance sampling
algorithm with weights 
\[
w\left(X^{n,1:M},X^{n,1:\tilde{M}}\right)=\frac{\rho_{\lambda}\left(d\theta,d\tilde{X}^{n,1:\tilde{M}}\right)L\left(\tilde{X}^{n,1:\tilde{M}},dX^{n,1:M}\right)}{\rho_{\lambda}\left(d\theta,dX^{n,1:M}\right)\prod_{i=1}^{\tilde{M}}\pi\left(\left.\tilde{X}^{n,i}\right|\theta\right)}.
\]
 The particular choice of the independent backward kernel $L\left(\tilde{X}^{n,1:\tilde{M}},dX^{n,1:M}\right)=\prod_{i=1}^{\tilde{M}}\pi\left(\left.\tilde{X}^{n,i}\right|\theta\right)$
yields tractable weights,
\[
w\left(X^{n,1:M},\tilde{X}^{n,1:\tilde{M}}\right)=\frac{M\sum_{i=1}^{\tilde{M}}e^{-\lambda d\left(S(\tilde{X}^{n,i}),S(Y^{n})\right)}}{\tilde{M}\sum_{i=1}^{M}e^{-\lambda d\left(S(X^{n,i})_{s},S(Y^{n})\right)}}.
\]
The importance sampling step takes the intuitive form of a ratio of
each contributions to the likelihood. It has been observed however
that this approach adds variance to weights. In fact our numerical
experiment show that this algorithm is not usable in practice on our
problem. We observe an important increase of the variance of the weights
that leads to a degeneracy after only a few step of the SMC sampler. 

\subsubsection{Gibbs sampling\label{subsec:Gibbs-sampling}}

Another approach based on similar idea is to use a Gibbs sampler,
hence not modifying the weights. We augment the space by some index
$k$ in the following way,

\[
\rho_{\lambda}\left(k,d\theta,dX^{n,1:M}\right)=\frac{1}{M}e^{-\lambda d\left(S(X^{n,k}),S(Y^{n})\right)}\prod_{i=1}^{M}\pi\left(dX_{i}^{n}\vert\theta\right).
\]
 It is clear that this distribution has the correct marginal (by summing
on the index $k$). To sample from this distribution we start by sampling
from the distribution of the index conditionally on $\theta$ and
$X^{n,1:M}$, then sample the $X^{n}$ conditionally on the rest.
A sample from the index's distribution is easily seen to be a draw
from and empirical distribution,

\[
\hat{k}\sim\sum_{i=1}^{M}w_{k}\delta_{k},\qquad\text{with}\qquad w_{k}=\frac{e^{-\lambda d\left(S(X^{n,k}),S(Y^{n})\right)}}{\sum_{j=1}^{M}e^{-\lambda d\left(S(X^{n,j}),S(Y^{n})\right)}}.
\]
 Hence $X^{n,\hat{k}}$ is kept fixed and it remains to sample from
the other conditional, 
\[
X^{n,2:\tilde{M}_x}\sim\prod_{i=2}^{\tilde{M}_x}\pi\left(\rm{d}\left.X^{n,i}\right|\theta\right),
\]
 where without loss of generality we put $\hat{k}=1$.

\begin{algorithm}[H]
\caption{Change of $M$ using Gibbs}
 \label{alg:ChangeM-Gibbs }

 \begin{description}     
\item[Input]$M_x$, $\tilde{M}_x$     
\item[For ]$i\in \lbrace 1, \cdots, N\rbrace$     
\begin{description}     
\item[a.] Sample an index  $k\vert X_{1:M_x},\theta \sim \sum_{j=1^M}\delta_{j}(k) W_j$     \item[b.] Sample $\tilde{M}_x-1$ new samples $X^n_{2:\tilde{M}_x}\sim \prod_{i=2}^{\tilde{M}_x}\pi(d X^n_i\vert \theta)$    
\end{description}
\item[End For]
\end{description}  
\end{algorithm}

The Gibbs sampling approach has the key advantage of not modifying
the weights and therefore introduces less variance on complicated
problems. As discussed above we observed, in practice, that for the
purpose of ABC it was not feasible to use importance sampling to choose
$M$. However as we will see in the numerical experiments Gibbs sampling
allows a control on the acceptance ratio.

\subsection{Adaptation to the sequence of inverse temperature}

\label{sec:inv-temp}

\cite{DelMoral2012} propose to chooses automatically the new window
in order to match the ESS of the weights to a fix value.

For the first few steps of the algorithm we set $M=1$. Choosing at
$t$ the parameter $\lambda_{t+1}$ to consider next is a matter of
solving a one dimensional equation $ESS(\lambda)=\tau N$ for a given
fixed $\tau$. In addition we know that the ESS is a decreasing function
of $\lambda$. Following \cite{Jasra2011} we propose to use a bi-section
search to solve the equation (see \cite{press2007numerical}). In
the case $M=1$ to change the weights at each new proposed $\lambda$
implies only multiplying the log weights by $\lambda$. The run-time
of the subroutine is negligible with respect to rest of the algorithm
as we can store the distances of each particle to the observed value.
In the case where $M>1$ the computation might be a bit more costly.
However we can keep track of the successive values of $\lambda$ and
use them to predict the next. We propose using a regression in log-scale
to predict the next outcome. We did not observe an important change
in the successive value of the ESS when using this coarse approximation.
In practice whenever we have reached the decision to increase $M$,
we swap to the strategy using the regression. We will typically fix
$\tau$ to be quite high ($\approx0.9$) to ensure a slow exploration
of the sequence of the posterior on what can be considered as difficult
problems.

\section{Numerical experiments\label{sec:Numerical-experiments}}

\subsection{Experiment 1}

We propose to use a variant of the experiment set in \cite{DelMoral2012}.
We use a mixture of two one dimensional Gaussian distributions. We
define the model to be,
\[
\pi(.\vert\theta)\equiv p\mathcal{N}(.;\mu_{1},\sigma_{1}^{2})+(1-p)\mathcal{N}(.;\mu_{2},\sigma_{2}^{2}),
\]
 to set aside any identification issues we assume the probability
known ($p=0.8$). The algorithm learns the rest of the parameters.
Our first experiment studies the case where the model is correctly
satisfied. The statistics are given by $S(x)=\left(x,x^{2},x^{3},x^{4},\textbf{1}_{x<-1},\textbf{1}_{x>2}\right).$
The distance is the one deduced by the Euclidean norm. The setup is
the following we take $n=90$, $\lambda=60$. We run the algorithm
with $N=3000$ particles. As a first measure, we compare the effect
of using a sampler with uniform kernel versus an exponential kernel.
We compare their precision for estimating the pseudo-posterior mean. 

\begin{figure}[H]
\caption{Boxplot for the error in parameter estimation }
\label{fig:boxplot}
\hfill{}%
\begin{minipage}[t]{0.4\columnwidth}%
\subfloat[parameter $\mu_{1}$]{\includegraphics[scale=0.25]{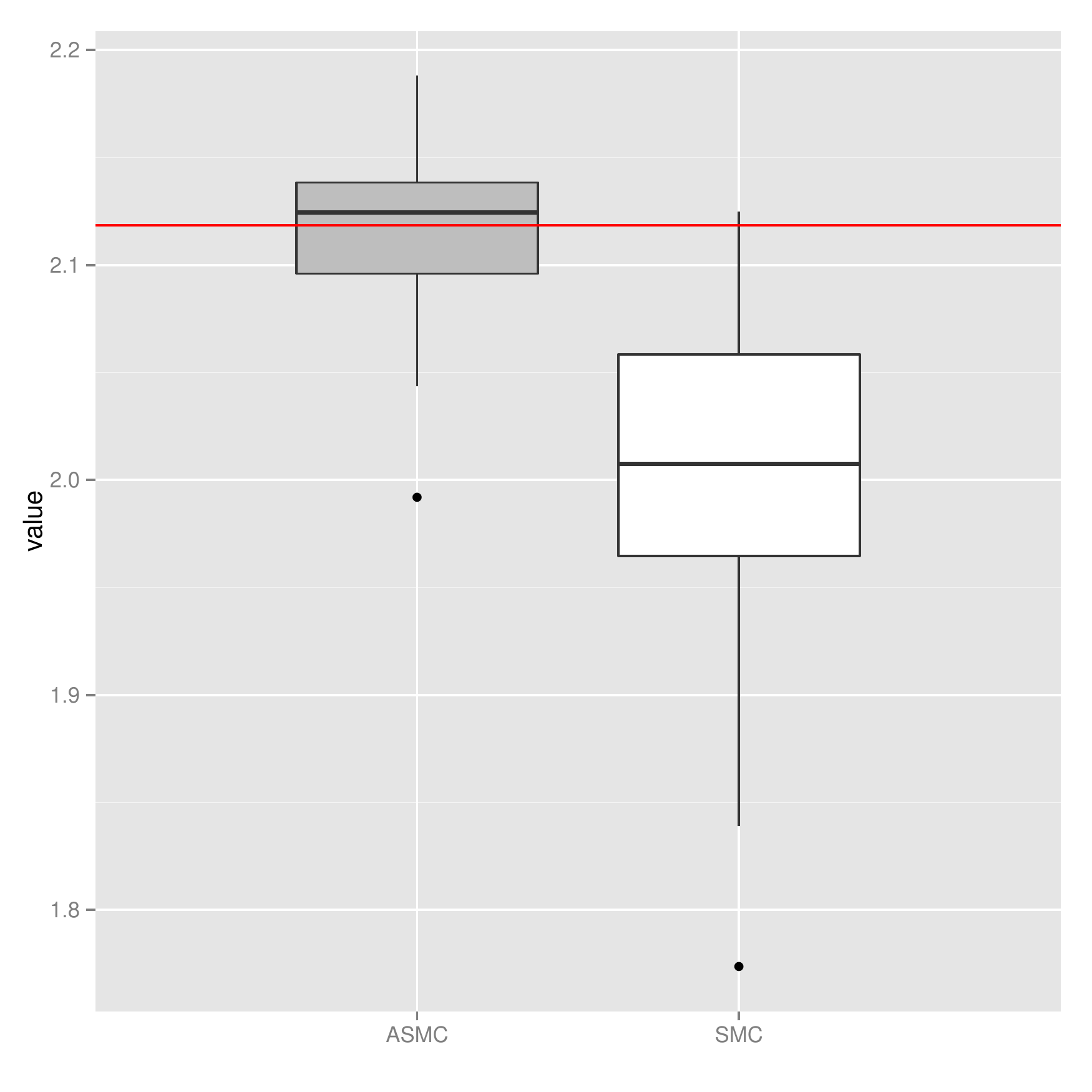}

}%
\end{minipage}\hfill{}%
\begin{minipage}[t]{0.4\columnwidth}%
\subfloat[parameter $\sigma_{1}$]{\includegraphics[scale=0.25]{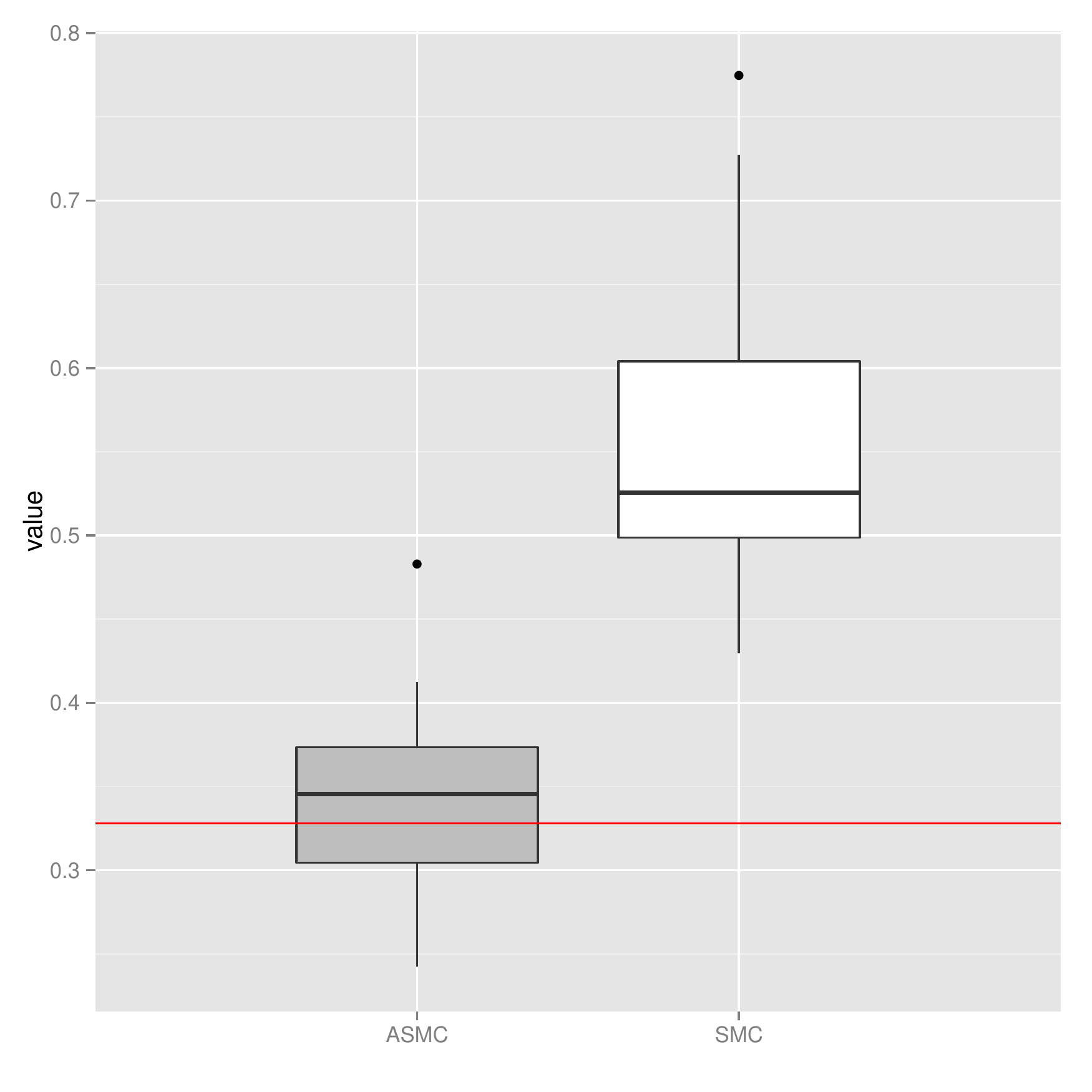}

}%
\end{minipage}\hfill{}

\hfill{}%
\begin{minipage}[t]{0.8\columnwidth}%
{\footnotesize{}Posterior mean of the parameters from both vanilla
SMC (adaptive only in the exploration of the sequence). The white
boxplot corresponds to the uniform kernel and the gray boxplot is
given by the exponential kernel. We obtain the red line by sampling
with ten times more particles. The two algorithm are run at constant
computational time}.%
\end{minipage}\hfill{}

\rule[0.5ex]{1\columnwidth}{1pt}
\end{figure}

There is already a substantial computational gain in using the exponential
kernel rather than the uniform. In the case of ABC this amounts to
weighing the particles rather than just throwing away the particles
that do not make a threshold. Intuitively at least, replacing the
hard threshold by a soft one, should reduce the variance. Furthermore
it seems from the experiments that exponentially weighted (EW) version
could explore more quickly the sequence of posteriors. This will serve
as a general justification for looking at the EW version from now
on. 

On the same experience we show the effect of increasing adaptively
the number of samples $M$. We compare two SMC samplers with exponential weights,
one with fixed $M=1$ and one that increases gradually $M$ according
to the rule that was defined in Section \ref{sec:inv-temp}. 

\begin{figure}[H]
\caption{Acceptance ratio (multiple samples)}

\label{fig:accept}\includegraphics[scale=0.3]{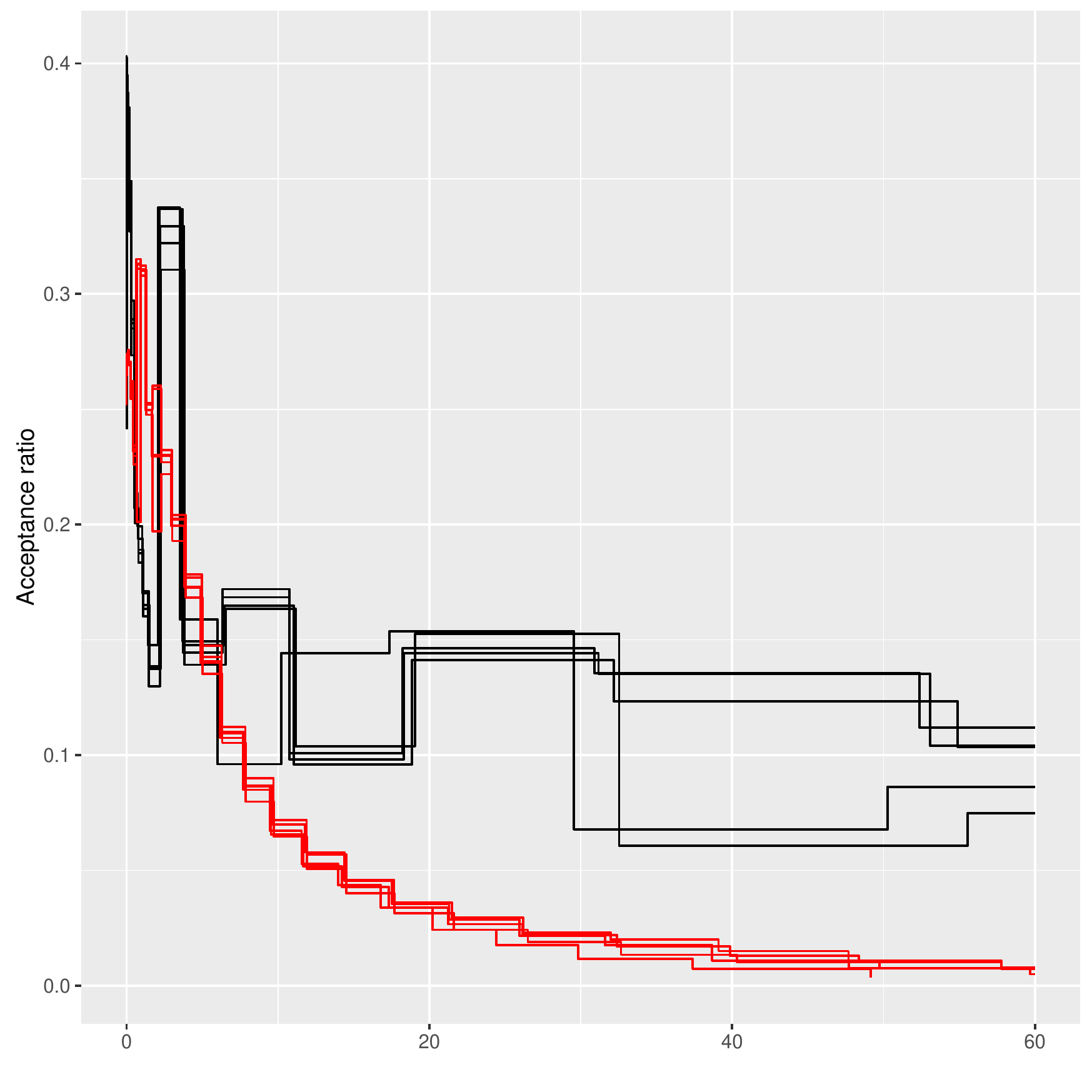}

\hfill{}%
\begin{minipage}[t]{0.8\columnwidth}%
{\footnotesize{}On the correctly specified data-set we show the evolution
of the acceptance ration with the increase of the inverse temperature.
The solid black line is given by the SMC-ABC with exponential kernel
under with an adaptive choice of of parameter $M$ indexing the number
of samples per particles. The red solid line is the corresponding
value for the vanilla EW SMC-ABC. Similar behavior was observed for
the case of the uniform kernel in \cite{DelMoral2012}. The targeted
acceptance ratio is 10\%.}%
\end{minipage}\hfill{}

\rule[0.5ex]{1\columnwidth}{1pt}
\end{figure}

It is already known that for fixed $M$ the value of the acceptance
ratio will decrease. In Figure \ref{fig:accept} we see that the effect
of using Gibbs sampling is to maintain the acceptance ratio at a given
level. We also tested the importance sampling algorithm, the effect
of changing the weights on this examples led to very high variance
to the point that the weight degenerated to $0$ ($ESS=1$). Several
additional point should be considered when using this algorithm. In
particular the increase of $M$ has also an effect on the memory of
the system as described at the beginning of Section \ref{sec:Monte-Carlo-algorithm}.
Different approach can be considered to treat this problem, in a technical
note \cite{Chopin2015} propose to store seeds used for the pseudo-random
number generator and to sample each trajectory each time it is needed.
We do not delve further on those problems as they go beyond the aim
of this paper.
In the rest of the experiements we will use exponential weights and the 
adaptive selection of the number of particles in $M$ presented in sub-Section \label{subsec:Gibbs-sampling}.

\subsection{Experiment 2}

We produce the same experiment as in the previous section, only this
time the model is mis-specified. We consider for a true model a mixture
of Gaussian with $3$ components. In the ABC pseudo-posterior we use
for $\left\{ \pi_{\theta},\theta\in\Theta\right\} $ the model described
in the previous experiment. We show the median and the maximum mean
square error of the statistics that are used in the pseudo posterior.
The MSE does not totally cancel even for the ABC that we develop as
there still exists a small bias. However we get a relatively small
value. To obtain the necessary bounds (replicating the example of
Section \ref{sec:Examples-of-bounds}) we truncate the observation
to the set $\left[-5,5\right]$ (this is equivalent to a change in
the statistics). We can therefore also compare the performance of
the algorithm with adaptive inverse temperature $\lambda$ (i.e. $\lambda$
is selected using by using the adaptive algorithm of Section \ref{subsec:Adaptive-bounds}).

\begin{figure}[h!]
\caption{Median and maximum MSE}
\label{fig:MSE}

\hfill{}%
\begin{minipage}[t]{0.4\columnwidth}%
\subfloat[Median MSE]{\includegraphics[scale=0.25]{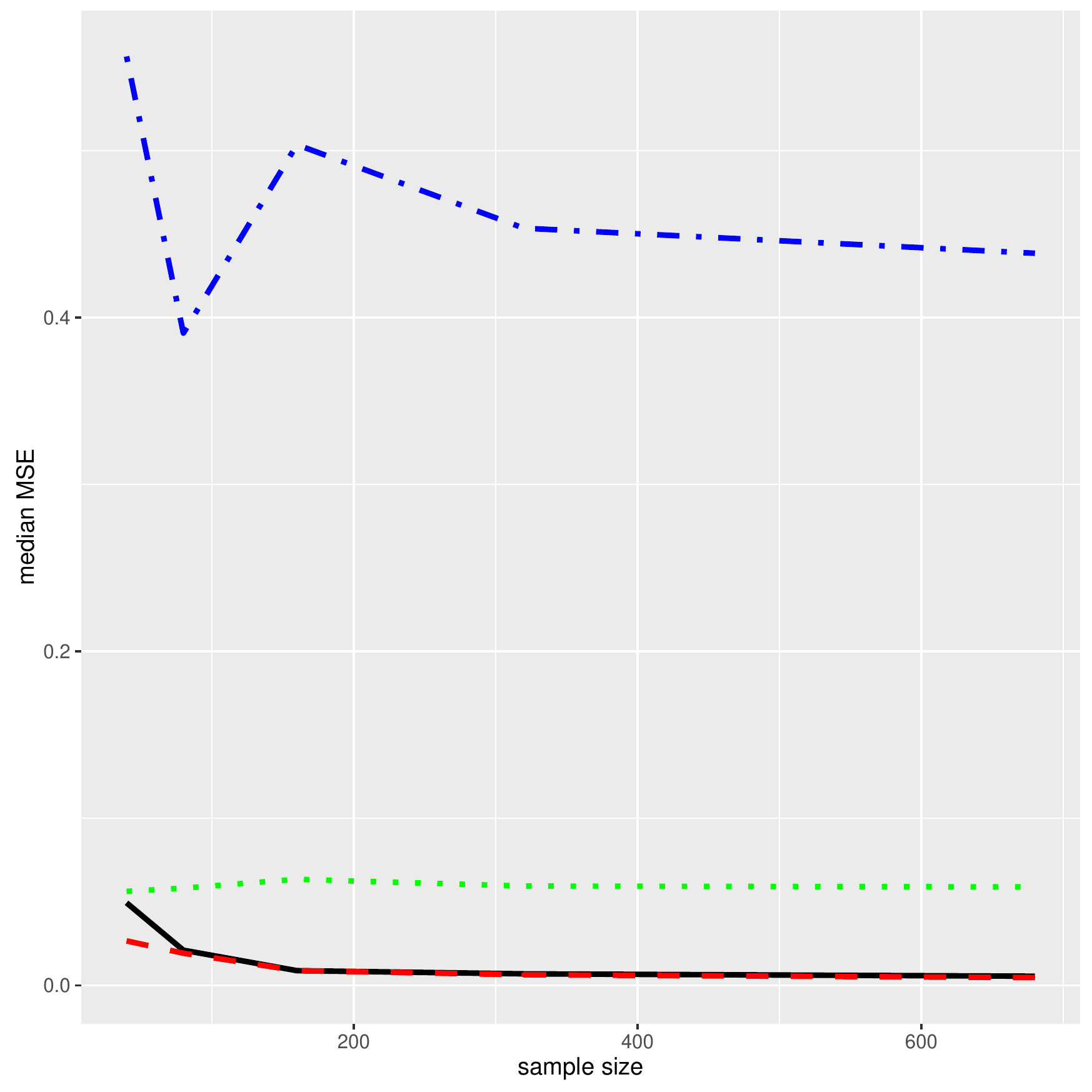}

}%
\end{minipage}\hfill{}%
\begin{minipage}[t]{0.4\columnwidth}%
\subfloat[Maximum MSE]{\includegraphics[scale=0.25]{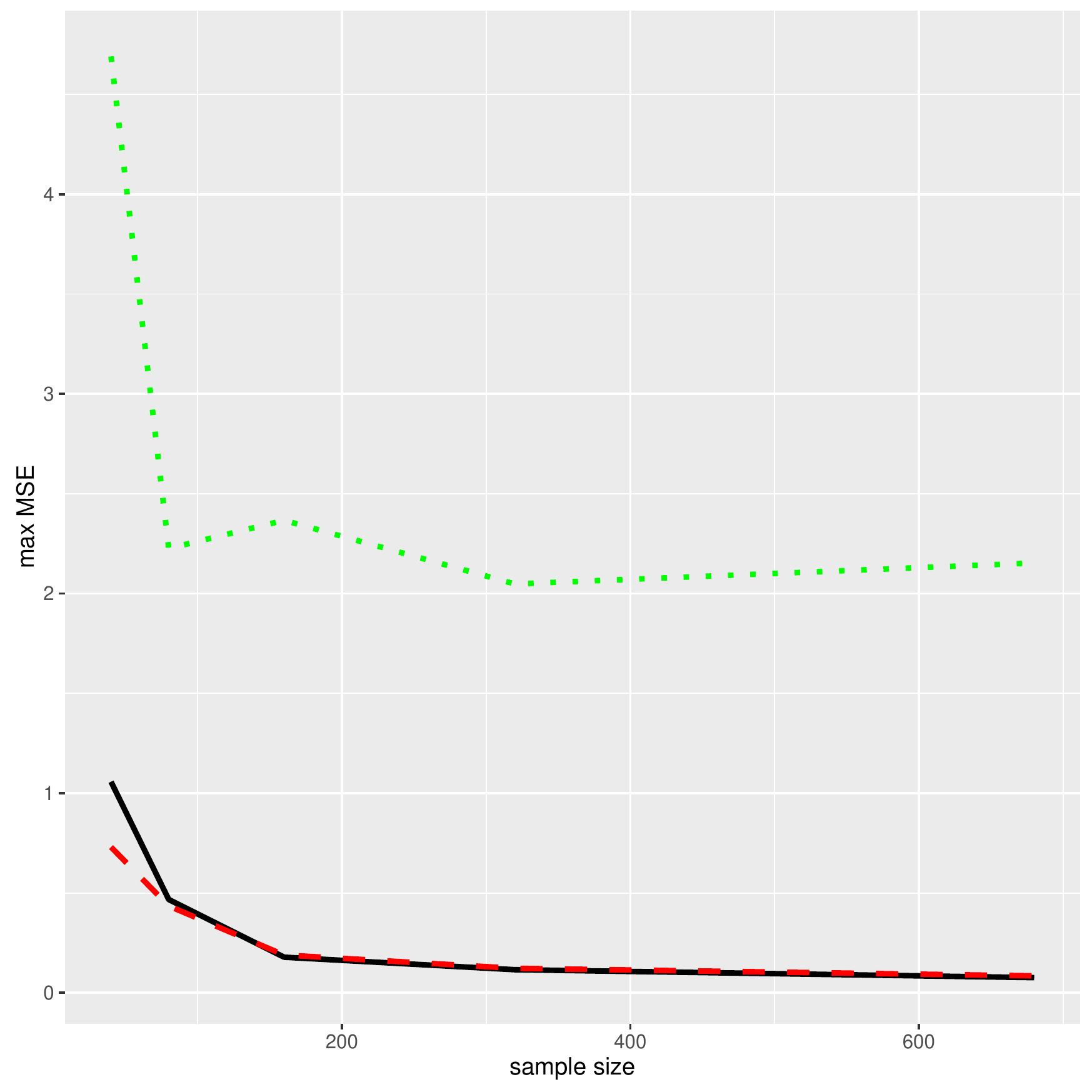}

}%
\end{minipage}\hfill{}

\hfill{}%
\begin{minipage}[t]{0.8\columnwidth}%
{\footnotesize{}We show the MSE corresponding to the algorithms described
in the previous section. The blue dotted-dash line is the original
SMC-ABC algorithm. Our adaptive SMC algorithm with $\lambda=90$ is
given by the red dashed line. The green dotted line is the MSE obtained
by using ``true'' samples from the false model. The solid black
line is the MSE of our adptive SMC with adaptive chosen inverse 
temperature (algorithm of Section \ref{subsec:Adaptive-bounds}).
All the computations are done at constant computational cost.}%
\end{minipage}\hfill{}

\rule[0.5ex]{1\columnwidth}{1pt}
\end{figure}

In Figure \ref{fig:MSE} we show as a function of the sample size
the decrease in MSE. The algorithms are also compared to the loss
that we would obtain from using the standard Bayesian posterior for
the wrong model (blue line). We see that our exponential weight and
adaptive exponential weight ABC (respectively solid black and dashed
red lines) perform well in this framework. 

We also give as an illustration the kind of empirical bounds one can
get using Section \ref{subsec:Empirical-bounds}. We show in Figure
\ref{fig:emp-bound} a bound on the Euclidean distance between the
moments of the statistics considered. The $y$-axis is given in logarithmic
scale. We see that for certain value of the inverse temperature the
bound on the generalized error is actually quite low. Recall however
that we can not use directly the bound for choosing $\lambda$ as
the bound is true in probability for individual values of the parameter.
One could however use a union bound or the adaptive ABC described
in Section \ref{subsec:Adaptive-bounds}. This is in fact very similar
in nature.

\begin{figure}[h!]
\caption{Empirical bound}

\includegraphics[scale=0.3]{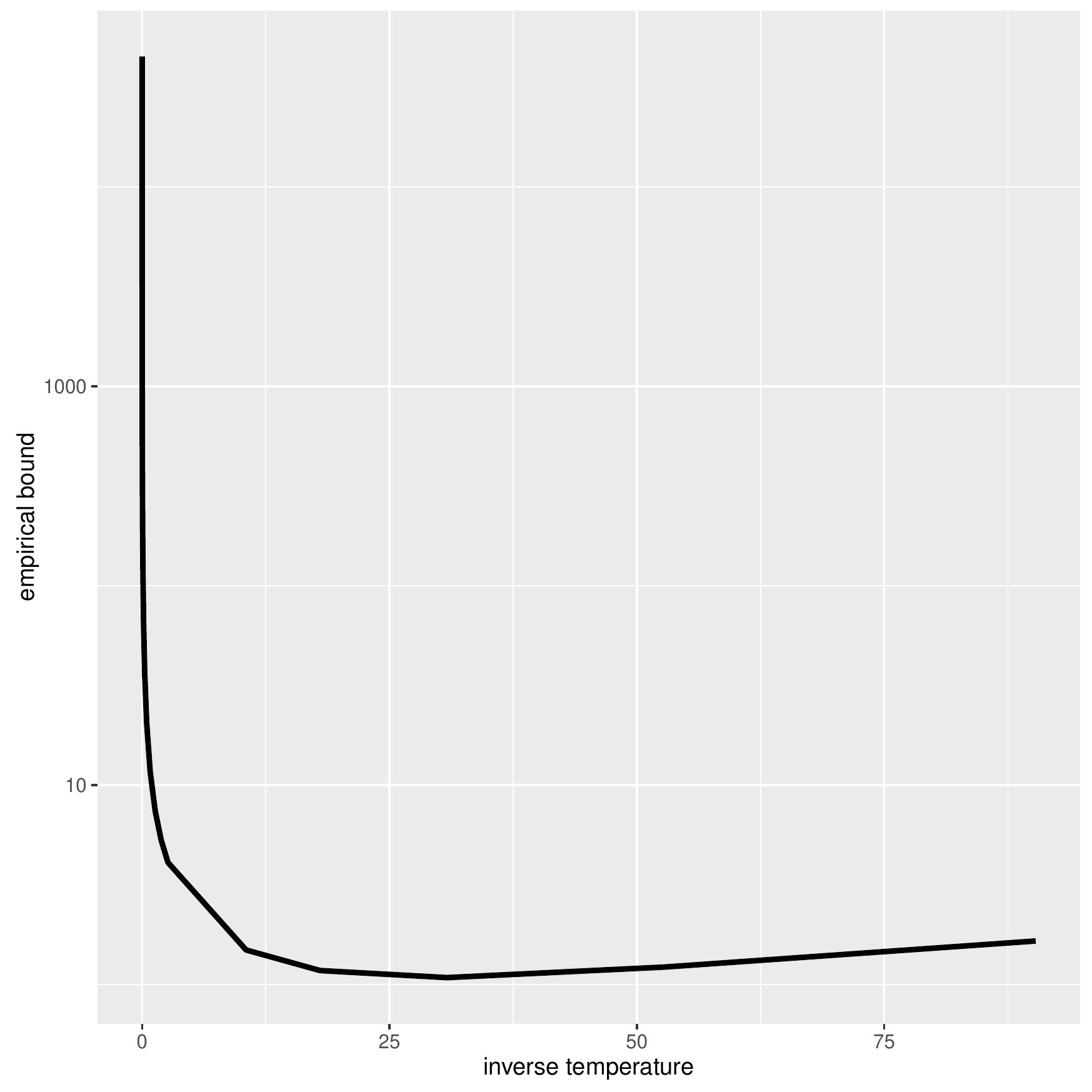}\label{fig:emp-bound}

\hfill{}%
\begin{minipage}[t]{0.8\columnwidth}%
{\footnotesize{}As an illustration we show an empirical bound obtained
for our algorithm under the misspecified setting. The bound is true
up to 95\% probability. The bound however as explained in Section
\ref{sec:Theoretical-bounds} is not true jointly (i.e. for the whole
value of $\lambda$ simultaneously). Here we can just observe the
order of magnitude, implying that the algorithm has significantly
learned from the data. }%
\end{minipage}\hfill{}

\rule[0.5ex]{1\columnwidth}{1pt}
\end{figure}

More details on the experiments can be found in appendix \ref{sec:appendix-implementation}.
\subsection{Experiment 3}

In this next experiment we use the same model, still under misspecification
only we propose this time to use as summary statistics the sequence
of indicator functions $\left(\textbf{1}_{x<t_{i}}\right)_{i=1}^{n}$
for a partition $\left(t_{i}\right)_{i}$ of the interval $[-5,5]$.
 The distance is taken to be the max-norm. Figure \ref{fig:indicator}
illustrates the effect of combining the indicators functions as statistics
and the distance based on the max-norm. The left panel shows that
the max error of the ABC algorithm is smaller, this however leads
to a uni-modal model (green curve on the right panel). On the other
hand the MCMC sampler estimates a model with two distinct modes. This
illustrates the importance of the choice of the different components
(distance and summary statistics) as they have a direct impact on
the final quantity one wants to control.

\begin{figure}[h!]
\caption{Indicators as a statistic }
\label{fig:indicator}

\hfill{}%
\begin{minipage}[t]{0.4\columnwidth}%
\subfloat[Error on the statistics]{\includegraphics[scale=0.25]{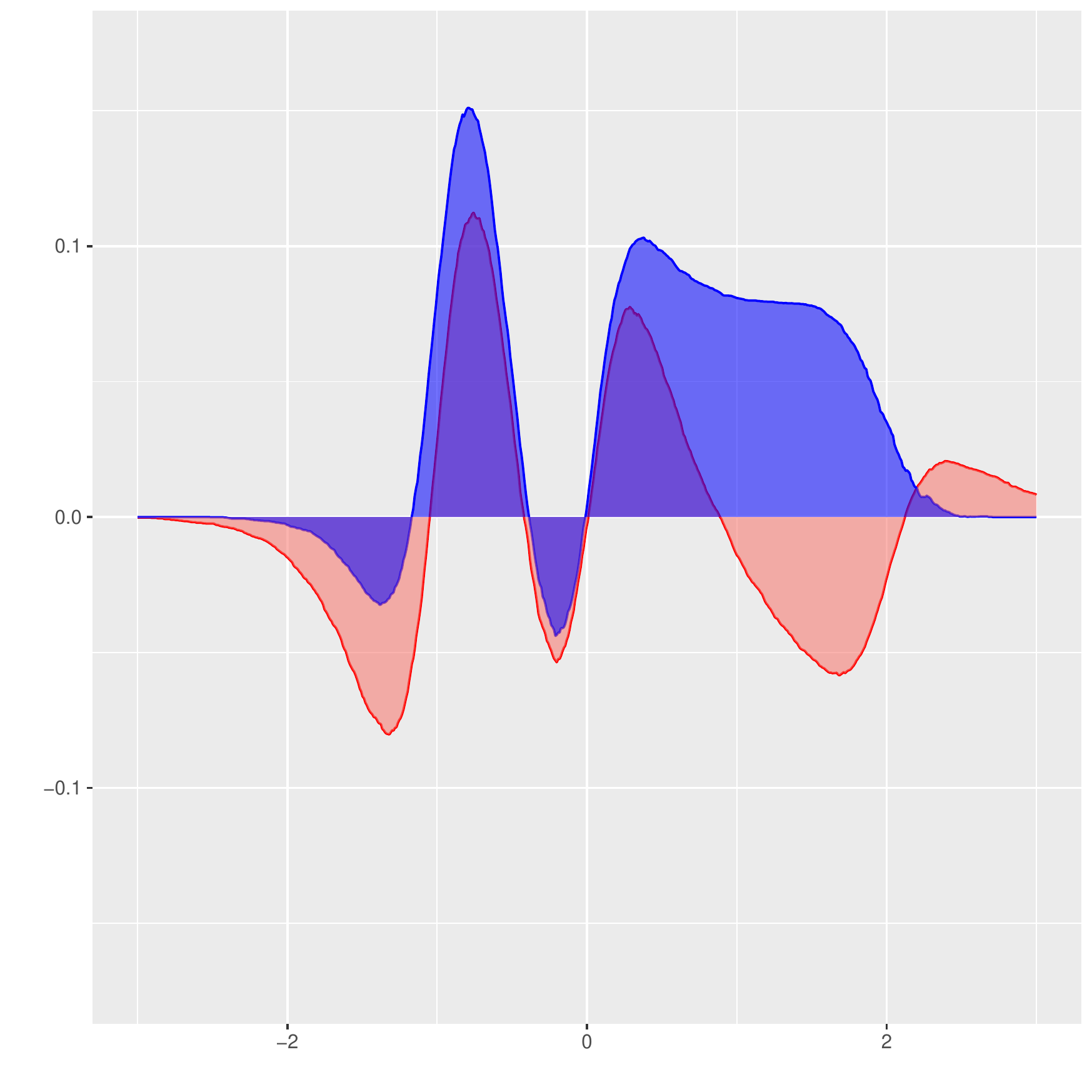}

}%
\end{minipage}\hfill{}%
\begin{minipage}[t]{0.4\columnwidth}%
\subfloat[Density, and predicted density ]{\includegraphics[scale=0.25]{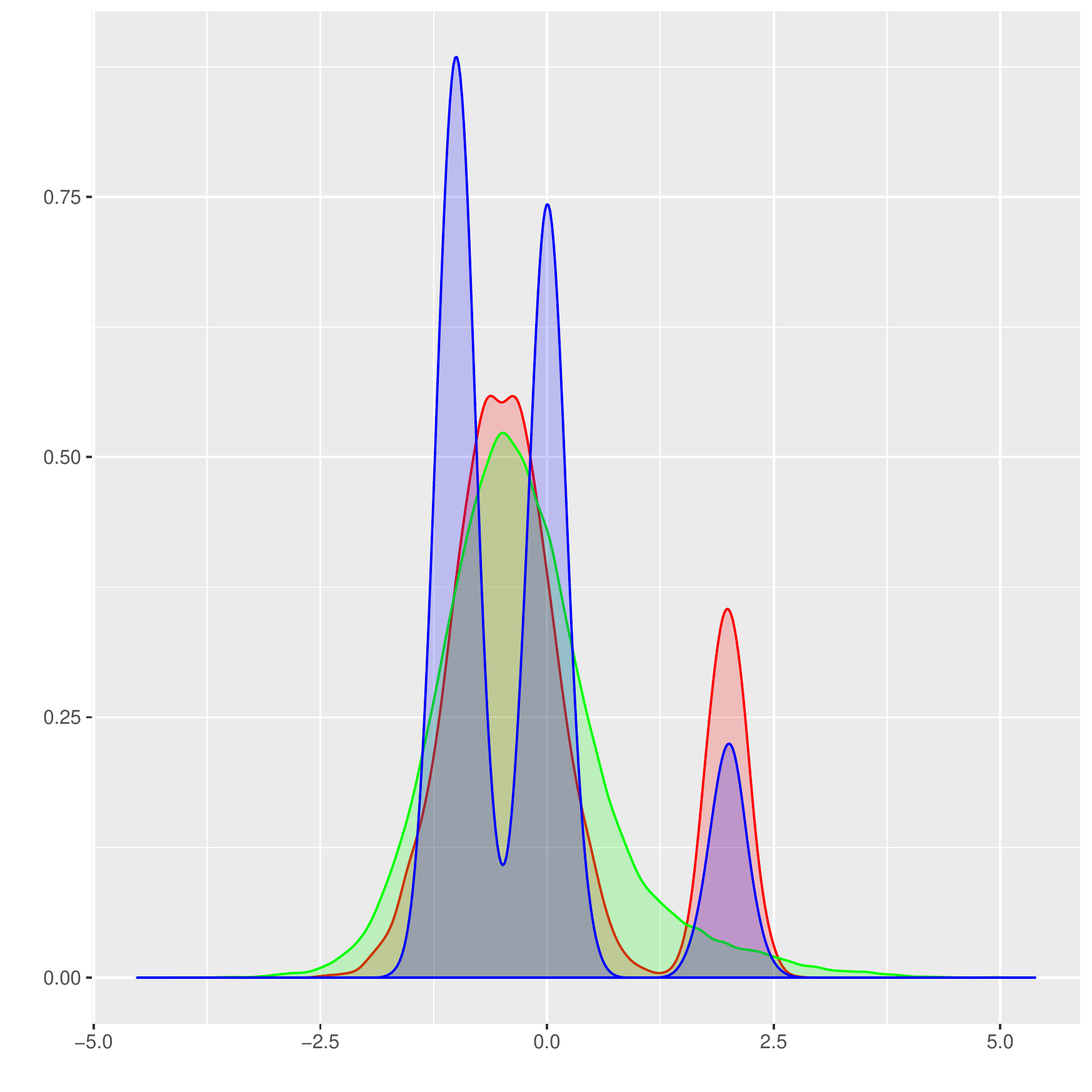}

}%
\end{minipage}\hfill{}

\hfill{}%
\begin{minipage}[t]{0.8\columnwidth}%
{\footnotesize{}We show some results for the sequence of statistics
$\left(\textbf{1}_{x<t_{i}}\right)_{i=1}^{n}$. On the left figure
we show the difference between the statistics obtained from our sampler
and the true statistics (averaged over the true model). In red we show
the error obtained from the ABC sampler and in blue the one obtained
from an MCMC sampler. On the right panel the estimated density of
the observation (blue) and the predicted data according to the MCMC
sampler (red ) and the ABC approximation (green).}%
\end{minipage}\hfill{}

\rule[0.5ex]{1\columnwidth}{1pt}
\end{figure}

\section{Proofs and supporting results\label{sec:Proofs-and-supporting}}

\subsection{Preliminary results}

We start be recalling the following well known formulae (e.g. \cite{Catoni2007} Chapter 1)

For any $\varrho\ll\pi$ and $h\in L^{1}(\varrho)$,
\[
\mathcal{K}(\varrho,\pi_{e^h})=\lambda\int h{\rm d}\varrho+\mathcal{K}(\varrho,\pi)+\log\int\exp(-h){\rm d}\pi.
\]
where $\pi_{e^h}:=\frac{e^{h}d\pi}{\pi\left[e^{h}\right]}$. 
 This implies two well known facts 
\begin{align}
 \pi_{e^h}& =\arg\min_{\varrho\in\mathcal{M}_{+}^{1}(\Theta)}\left\{ \int h{\rm d}\varrho+\mathcal{K}(\varrho,\pi)\right\} ,\label{eq:var1}\\
-\log\int\exp(-h){\rm d}\pi & =\min_{\varrho\in\mathcal{M}_{+}^{1}(\Theta)}\left\{ \int h{\rm d}\varrho+\mathcal{K}(\varrho,\pi)\right\} .\label{eq:var2}
\end{align}

We will use those two equations repeatedly in what follows.

\subsubsection*{Notation specific to the section}

In this section we will use the following shorthand notation for any
probability measure $\nu_{1}$,$\nu_{2}$, 
\[
w_{d}(\nu_{1},\nu_{2})=\int d(S(x),S(y))\nu_{1}(dx)\nu_{2}(dy).
\]

\subsection{Proof of Proposition \ref{prop:Emp}\label{subsec:prop:emp}}

The proofs of this section follow closely the techniques used in \cite{Catoni2007}.
\begin{lem}
\label{lem:unionbound}Under Assumption A\ref{assu:Hoeffding} for
any $\varrho\in\mathcal{M}_{+}^{1}$, $\lambda\in\mathcal{I}$, $n\geq1$
and $\epsilon>0$ jointly with probability at least $1-\epsilon$,

\begin{align}
\omega_{d}(\varrho,\mathbb{P})\leq & \omega_{d}\left(\varrho,\delta_{Y^{n}}\right)+\frac{1}{\lambda}\mathcal{K}(\varrho,\pi_{\theta}\pi)+\frac{f(n,\lambda)}{\lambda}+\frac{1}{\lambda}\log\frac{2}{\epsilon}\label{eq:empirical2}\\
\omega_{d}\left(\varrho,\delta_{Y^{n}}\right) & \leq\omega_{d}(\varrho,\mathbb{P})+\frac{1}{\lambda}\mathcal{K}(\varrho,\pi_{\theta}\pi)+\frac{f(n,\lambda)}{\lambda}+\frac{1}{\lambda}\log\frac{2}{\epsilon}.\nonumber 
\end{align}
\end{lem}
\begin{proof}
We start by making successive use of the equality in (\ref{eq:var1})
on the first line, Markov's inequality on the second and Assumption
A\ref{assu:Hoeffding} on the third. For any $\eta>0$

\begin{align*}
\mathbb{P}\left(\sup_{\varrho\in\mathcal{M}_{1}^{+}(\Theta\times\mathcal{X}^{n})}\left\{ \lambda\varrho(D_{n}^{S}-\mathbb{P}D_{n}^{S})-\mathcal{K}(\varrho,\pi_{\theta}\pi)+\eta\right\} \geq0\right) & =\mathbb{P}\left(\log m_{X}\left[e^{\lambda(D_{n}^{S}-\mathbb{P}D_{n}^{S})+\eta}\right]\geq0\right)\\
 & \leq m_{X}\otimes\mathbb{P}\left[e^{\lambda(D_{n}^{S}-\mathbb{P}D_{n}^{S})+\eta}\right]\\
 & \leq e^{f(n,\lambda)+\eta}.
\end{align*}

The inequality is true for any $\eta$ in particular choose $\eta=-\log\frac{1}{\epsilon}-f(n,\lambda)$,

\[
\mathbb{P}\left(\sup_{\varrho\in\mathcal{M}_{1}^{+}(\Theta\times\mathcal{X}^{n})}\left\{ \lambda\varrho(D_{n}^{S}-\mathbb{P}D_{n}^{S})-\mathcal{K}(\varrho,\pi_{\theta}\pi)\right\} -\log\frac{1}{\epsilon}-f(n,\lambda)\geq0\right)\leq\epsilon.
\]

Therefore for any $\varrho\in\mathcal{M}_{+}^{1}(\Theta\times\mathcal{X}^{n})$
with probability at least $1-\epsilon$ ,

\begin{align}
\lambda\varrho(D_{n}^{S}-\mathbb{P}D_{n}^{S}) & \leq\mathcal{K}(\varrho,\pi_{\theta}\pi)+f(n,\lambda)+\log\frac{1}{\epsilon}.\label{eq:prob-bound}
\end{align}

We note that owing to the symmetry of Assumption A\ref{assu:Hoeffding}
we can apply the same reasoning to $-D_{n}^{S}$. Hence we get with
probability at least $1-\epsilon$,
\begin{align}
\lambda\varrho(\mathbb{P}D_{n}^{S}-D_{n}^{S}) & \leq\mathcal{K}(\varrho,\pi_{\theta}\pi)+f(n,\lambda)+\log\frac{1}{\epsilon}.\label{eq:emplemma}
\end{align}

We conclude the proof by applying a union bound to both equations,
rearranging the terms and using the definition of $w_{d}$. 
\end{proof}
The proof of the empirical bound (Proposition \ref{prop:Emp}) can
be obtain from equation (\ref{eq:emplemma}). Using Assumption A\ref{assu:distance}
and Jensen's inequality with probability at least $1-\epsilon$, 
\[
d(\rho(S),\mathbb{P}(S))\leq\varrho(\mathbb{P}D_{n}^{S})\leq\rho\left(D_{n}^{S}\right)+\frac{1}{\lambda}\left(\mathcal{K}(\varrho,\pi_{\theta}\pi)+f(n,\lambda)+\log\frac{1}{\epsilon}\right).
\]

This proves the first part of proposition \ref{prop:Emp}, the second
part is application of equation \ref{eq:var2}, by taking the infimum
of the upper bound in $\rho\in\mathcal{M}_{+}^{1}$ we obtain the
desired result.

\subsection{Proof of Lemma \ref{lem:bound}\label{subsec:Proof-of-lemma1}}

We start from the following lemma, a simple extension of Lemma \ref{lem:unionbound},
\begin{lem}
\label{lem:lemma_oracle1}Under Assumption A\ref{assu:Hoeffding},
for any $\lambda\in\mathcal{I}$, $n\geq1$and $\epsilon>0$ we have
with probability at least $1-\epsilon$,

\[
\omega_{d}(\varrho_{\lambda},\mathbb{P})\leq\inf_{\varrho\in\mathcal{M}_{+}^{1}(\Theta\times\mathcal{X}^{n})}\left\{ \omega_{d}(\varrho,\mathbb{P})+\frac{2}{\lambda}\mathcal{K}(\varrho,\pi_{\theta}\pi)\right\} +2\frac{f(\lambda,n)}{\lambda}+\frac{2}{\lambda}\log\frac{2}{\epsilon}.
\]
\end{lem}
\begin{proof}
By applying the inequalities of Lemma \ref{lem:unionbound} to the
case where $\varrho$ is the ABC measure and noticing that by equation
\ref{eq:var1} this measure is solution of the variational problem
minimizing the empirical bound \ref{eq:empirical2} we get jointly
with probability at least $1-\epsilon$,
\begin{align*}
\omega_{d}(\varrho_{\lambda},\mathbb{P})\leq & \inf_{\varrho\in\mathcal{M}_{+}^{1}(\Theta\times\mathcal{X}^{n})}\left\{ \omega_{d}\left(\varrho,\delta_{Y^{n}}\right)+\frac{1}{\lambda}\mathcal{K}(\varrho,\pi_{\theta}\pi)\right\} +\frac{f(\lambda,n)}{\lambda}+\frac{1}{\lambda}\log\frac{2}{\epsilon}\\
\omega_{d}\left(\varrho_{\lambda},\delta_{Y^{n}}\right) & \leq\omega_{d}(\varrho_{\lambda},\mathbb{P})+\frac{1}{\lambda}\mathcal{K}(\varrho_{\lambda},\pi_{\theta}\pi)+\frac{f(\lambda,n)}{\lambda}+\frac{1}{\lambda}\log\frac{2}{\epsilon}.
\end{align*}

We get the desired result by combining both equations. 
\end{proof}
The proof of Lemma \ref{lem:bound} starts from noticing that under
Assumption A\ref{assu:distance} we get from Lemma \ref{lem:lemma_oracle1}
by Jensen's inequality, with probability at least $1-\epsilon$, 
\[
d(\varrho_{\lambda}(S),\mathbb{P}S)\leq\omega_{d}(\varrho_{\lambda},\mathbb{P})\leq\inf_{\varrho\in\mathcal{M}_{+}^{1}(\Theta\times\mathcal{X}^{n})}\left\{ \omega_{d}(\varrho,\mathbb{P})+\frac{2}{\lambda}\mathcal{K}(\varrho,\pi_{\theta}\pi)\right\} +2\frac{f(\lambda,n)}{\lambda}+\frac{2}{\lambda}\log\frac{2}{\epsilon}.
\]

Recall that the oracle parameter $\theta^{\star}$ is defined as the
minimizer of $\theta\mapsto d(\pi_{\theta}(S),\mathbb{P}S)$ from
definition \ref{def:The-oracle-parameter}. The infimum over all measures
can be upper bounded by using the following parametric family,

\[
\varrho_{\theta^{\star},\delta}(dX^{n},d\theta)=\frac{\pi_{\theta}(dX^{n})\pi(d\theta)\textbf{1}_{\Vert\theta-\theta^{\star}\Vert<\delta}}{\pi(\Vert\theta-\theta^{\star}\Vert<\delta)}.
\]

we can therefore weaken the bound, with probability at least $1-\epsilon$,

\[
d(\varrho_{\lambda}(S),\mathbb{P}S)\leq\inf_{\delta\in(0,\bar{\delta})}\left\{ \omega_{d}(\varrho_{\theta^{\star},\delta},\mathbb{P})+\frac{2}{\lambda}\mathcal{K}(\varrho_{\theta^{\star},\delta},\pi_{\theta}\pi)\right\} +2\frac{f(\lambda,n)}{\lambda}+\frac{2}{\lambda}\log\frac{2}{\epsilon}.
\]

Direct calculation yields that the KL divergence is expressed as,

\begin{equation}
\mathcal{K}(\varrho_{\theta^{\star},\delta},\pi_{\theta}\pi)=-\log\pi\left(\Vert\theta-\theta^{\star}\Vert<\delta\right).\label{eq:small-ball}
\end{equation}

It remains to deal with the term $\omega_{d}(\varrho_{\theta^{\star},\delta},\mathbb{P})$,
we have from the triangle inequality,

\[
d(S(X^{n}),S(Y^{n}))\leq d(S(X^{n}),\pi_{\theta}S)+d(\pi_{\theta}S,\pi_{\theta^\star} S)+d(\pi_{\theta^\star}S,\mathbb{P}S)+d(\mathbb{P}S,S(Y^{n})).
\]

We use this inequality in the definition of $w_{d}$,

\begin{align*}
\omega_{d}(\varrho_{\theta^{\star},\delta},\mathbb{P}) & =\int d(S(X^{n}),S(Y^{n}))\mathbb{P}(dY^{n})\frac{\pi_{\theta}(dX^{n})\pi(d\theta)\textbf{1}_{\Vert\theta-\theta^{\star}\Vert<\delta}}{\pi(\Vert\theta-\theta^{\star}\Vert<\delta)}\\
 & \leq \sup_{\theta:\Vert\theta-\theta^{\star}\Vert<\delta}\pi_\theta d(S,\pi_{\theta}S)+\int d(\pi_{\theta}S,\pi_{\theta^\star} S)\varrho_{\theta,\delta}(\rm{d}\theta)+d(\pi_{\theta^\star}S,\mathbb{P}S)+\mathbb{P}d(\mathbb{P}S,S))
\end{align*}

Putting everything together yields the correct result.

\subsection{Proof of Theorem \ref{thm:theta}\label{subsec:thm-theta}}

We start by using Assumption A\ref{assu:inject} and the triangle
inequality
\[
\varrho_{\lambda}\left\{ \left\Vert \theta-\theta^{\star}\right\Vert \right\} \leq K\varrho_{\lambda}\left\{ d\left(\pi_{\theta}S,\pi_{\theta^{\star}}S\right)\right\} \leq K\varrho_{\lambda}\left\{ d\left(\pi_{\theta}S,S\right)\right\} +K\varrho_{\lambda}\left\{ d\left(S,\mathbb{P}S\right)\right\} +Kd\left(\pi_{\theta^{\star}}S,\mathbb{P}S\right).
\]
 The second term on the right hand side is bounded above by Lemma
\ref{lem:bound}, the last term is the oracle risk. We concentrate
on the bound for $\varrho_{\lambda}\left\{ d\left(\pi_{\theta}S,S\right)\right\} $.

We start by recalling the following well known fact, fix a probability
measure $\pi$ then 
\[
\forall\varrho\in\mathcal{M}_{+}^{1},\lambda\in\mathbb{R}_{+\star}\quad\varrho(h)\leq\frac{1}{\lambda}\log\pi\left(e^{\lambda h}\right)+\frac{1}{\lambda}\mathcal{K}(\varrho,\pi).
\]
 We apply this inequality to $\varrho_{\lambda}\left\{ d\left(\pi_{\theta}S,S\right)\right\} $
yielding
\begin{align*}
\varrho_{\lambda}\left\{ d\left(\pi_{\theta}S,S\right)\right\}  & \leq\frac{1}{\lambda}\log\pi\pi_{\theta}\left(e^{\lambda d\left(\pi_{\theta}S,S\right)}\right)+\frac{1}{\lambda}\mathcal{K}\left(\varrho_{\lambda},\pi\pi_{\theta}\right)\\
 & \leq\frac{1}{\lambda}\log\pi\pi_{\theta}\left(e^{\lambda d\left(\pi_{\theta}S,S\right)}\right)-\frac{1}{\lambda}\log\pi\left(e^{-\lambda D_{n}^{S}}\right)-\varrho_{\lambda}\left(D_{n}^{S}\right),
\end{align*}
where the last line comes from an expansion of the KL term. By positivity
of the distance and by multiplication by $e^{\lambda\pi_{\theta}d\left(\pi_{\theta}S,S\right)}$
we get 
\[
\varrho_{\lambda}\left\{ d\left(\pi_{\theta}S,S\right)\right\} \leq\frac{1}{\lambda}\log\pi\left\{ \pi_{\theta}\left(e^{\lambda\left\{ d\left(\pi_{\theta}S,S\right)-\pi_{\theta}d\left(\pi_{\theta}S,S\right)\right\} }\right)e^{\lambda\pi_{\theta}d\left(\pi_{\theta}S,S\right)}\right\} -\frac{1}{\lambda}\log\pi\left(e^{\lambda D_{n}^{S}}\right).
\]
 Under Assumption A\ref{assu:Hoeffding-model} we can bound the first
part

\begin{align*}
\varrho_{\lambda}\left\{ d\left(\pi_{\theta}S,S\right)\right\}  & \leq\frac{\tilde{f}(n,\lambda)}{\lambda}+\frac{1}{\lambda}\log\pi\left\{ e^{\lambda\pi_{\theta}d\left(\pi_{\theta}S,S\right)}\right\} -\frac{1}{\lambda}\log\pi\left(e^{\lambda D_{n}^{S}}\right)\\
 & \leq\frac{\tilde{f}(n,\lambda)}{\lambda}+\sup_{\theta\in\Theta}\pi_{\theta}d\left(\pi_{\theta}S,S\right)+\inf_{\varrho\in\mathcal{M}_{1}^{+}}\left\{ \varrho(D_{n}^{S})+\frac{1}{\lambda}\mathcal{K}(\varrho,\pi)\right\} .
\end{align*}
 It remains to treat the last part, notice that we can make use of
equation (\ref{eq:prob-bound}) with probability a least $1-\epsilon$
we have 
\[
\varrho_{\lambda}\left\{ d\left(\pi_{\theta}S,S\right)\right\} \leq\frac{\tilde{f}(n,\lambda)}{\lambda}+\sup_{\theta\in\Theta}\pi_{\theta}d\left(\pi_{\theta}S,S\right)+\inf_{\varrho\in\mathcal{M}_{1}^{+}}\left\{ \varrho(\mathbb{P}D_{n}^{S})+\frac{2}{\lambda}\mathcal{K}(\varrho,\pi_{\theta}\pi)\right\} +\frac{1}{\lambda}\log\frac{1}{\epsilon}.
\]

We can now use the developments of Lemma \ref{lem:bound}, putting
$\varrho_{\theta^{\star},\delta}(dX^{n},d\theta)=\frac{\pi_{\theta}(dX^{n})\pi(d\theta)\textbf{1}_{\Vert\theta-\theta^{\star}\Vert<\delta}}{\pi(\Vert\theta-\theta^{\star}\Vert<\delta)}$

\begin{align*}
\inf_{\varrho\in\mathcal{M}_{1}^{+}}\left\{ \varrho(\mathbb{P}D_{n}^{S})+\frac{2}{\lambda}\mathcal{K}(\varrho,\pi)\right\}  & \leq\inf_{\delta\in(0,\bar{\delta})}\left\{ \omega_{d}(\varrho_{\theta^{\star},\delta},\mathbb{P})+\frac{2}{\lambda}\mathcal{K}(\varrho_{\theta^{\star},\delta},\pi_{\theta}\pi)\right\} \\
 & \leq \sup_{\theta:\Vert\theta-\theta^{\star}\Vert<\delta}\pi_{\theta}d(S,\mathbb{\pi_{\theta}}S)+\delta L+d(\pi_{\theta^{\star}}S,\mathbb{P}S)+\mathbb{P}d(\mathbb{P}S,S)-\log\pi\left(\Vert\theta-\theta^{\star}\Vert<\delta\right),
\end{align*}
 where we have used equation (\ref{eq:small-ball}) and equation (\ref{eq:triangle-ineq}).
We get the desired result by putting everything together and using
a union bound.

\subsection{Proof of Lemma \ref{lem:oracle-adap}}

Let $\mathcal{I}\subset\mathbb{R}_{+}$be the possible range of the
inverse temperature parameter. Define a prior measure $\nu$ on $(\mathcal{I},\mathcal{B}(\mathcal{I}))$.
We follow the lines of the proof of Lemma \ref{lem:bound} only the
variational procedure is now taken on measures on $(\Theta\times\mathcal{X}^{n}\times\mathcal{I})$.
\begin{lem}
\label{lem:unionbound-adap}Under Assumption A\ref{assu:Hoeffding}
for any $\mu\in\mathcal{M}_{+}^{1}(\Theta\times\mathcal{X}^{n}\times\mathcal{I})$,
$n\geq1$ and $\epsilon>0$ jointly with probability at least $1-\epsilon$,
\begin{align*}
\left.\begin{array}{c}
\mu\left\{ \lambda(\mathbb{P}D_{n}^{S}-D_{n}^{S})\right\} \\
\mu\left\{ \lambda(D_{n}^{S}-\mathbb{P}D_{n}^{S})\right\} 
\end{array}\right\}  & \leq\mu\left[f(n,\lambda)\right]+\mathcal{K}\left(\mu,\left[\pi_{\theta}\pi\right]\otimes\nu\right)+\log\frac{2}{\epsilon}.
\end{align*}
\end{lem}
\begin{proof}
Recall the starting point of the proof of Lemma \ref{eq:emplemma},
and use as before equations \ref{eq:var1}, \ref{eq:var2}, Markov's
inequality and Assumption A\ref{assu:Hoeffding},

\begin{align*}
\mathbb{P}\left(\sup_{\mu\in\mathcal{M}_{1}^{+}(\Theta\times\mathcal{X}^{n}\times\mathcal{I})}\left\{ \mu\left\{ \lambda(D_{n}^{S}-\mathbb{P}D_{n}^{S})+\eta(\lambda)\right\} -\mathcal{K}\left(\mu,\left[\pi_{\theta}\pi\right]\otimes\nu\right)\right\} \geq0\right) & =\mathbb{P}\left(\log\nu\otimes m_{X}\left[e^{\lambda(D_{n}^{S}-\mathbb{P}D_{n}^{S})+\eta}\right]\geq0\right)\\
 & \leq\nu\otimes m_{X}\mathbb{\otimes P}\left[e^{\lambda(D_{n}^{S}-\mathbb{P}D_{n}^{S})+\eta(\lambda)}\right]\\
 & \leq\nu\left[e^{f(n,\lambda)+\eta(\lambda)}\right]
\end{align*}

by choosing $\eta(\lambda)=-\log\frac{1}{\epsilon}-f(n,\lambda)$
we get, 
\[
\mathbb{P}\left(\sup_{\mu\in\mathcal{M}_{1}^{+}(\Theta\times\mathcal{X}^{n}\times\mathcal{I})}\left\{ \mu\left\{ \lambda(D_{n}^{S}-\mathbb{P}D_{n}^{S})-f(n,\lambda)\right\} -\mathcal{K}\left(\mu,\left[\pi_{\theta}\pi\right]\otimes\nu\right)\right\} \geq\log\frac{1}{\epsilon}\right)\leq\epsilon.
\]

Hence for any $\mu\in\mathcal{M}_{1}^{+}(\Theta\times\mathcal{X}^{n}\times\mathcal{I})$
with probability at least $1-\epsilon$, 

\[
\mu\left\{ \lambda(D_{n}^{S}-\mathbb{P}D_{n}^{S})\right\} \leq\mu\left[f(n,\lambda)\right]+\mathcal{K}\left(\mu,\left[\pi_{\theta}\pi\right]\otimes\nu\right)+\log\frac{1}{\epsilon}.
\]

By symmetry we get for any $\mu\in\mathcal{M}_{1}^{+}(\Theta\times\mathcal{X}^{n}\times\mathcal{I})$
with probability at least $1-\epsilon$,

\[
\mu\left\{ \lambda(\mathbb{P}D_{n}^{S}-D_{n}^{S})\right\} \leq\mu\left[f(n,\lambda)\right]+\mathcal{K}\left(\mu,\left[\pi_{\theta}\pi\right]\otimes\nu\right)+\log\frac{1}{\epsilon}.
\]
 We conclude using a union bound.
\end{proof}
In what follows we will restrict ourselves on a specific kind of factorizable
measures that allow us to further perform calculations. We let $\mu=\rho\otimes\xi\in\mathcal{M}_{1}^{+}(\Theta\times\mathcal{X}^{n})\times\mathcal{M}_{1}^{+}(\mathcal{I})\subset\mathcal{M}_{1}^{+}(\Theta\times\mathcal{X}^{n}\times\mathcal{I})$,
and deduce from lemma \ref{lem:unionbound-adap} the following result,
\begin{lem}
Under Assumption A\ref{assu:Hoeffding} for any $\rho\in\mathcal{M}_{1}^{+}(\Theta\times\mathcal{X}^{n})$,
$\xi\in\mathcal{M}_{1}^{+}(\mathcal{I})$, $n\geq1$ and $\epsilon>0$
jointly with probability at least $1-\epsilon$,
\[
\rho_{\hat{\xi}(\lambda)}\left[\mathbb{P}D_{n}^{S}\right]\leq\inf_{\xi\in\mathcal{F}}\left\{ \inf_{\rho\in\mathcal{M}_{1}^{+}(\Theta\times\mathcal{X}^{n})}\left\{ \rho(\mathbb{P}D_{n}^{S})+\frac{2}{\xi(\lambda)}\mathcal{K}\left(\rho,\pi_{\theta}\pi\right)\right\} +\frac{2}{\xi(\lambda)}\left(\xi\left[f(n,\lambda)\right]+\mathcal{K}\left(\xi,\nu\right)+\log\frac{2}{\epsilon}\right)\right\} .
\]
\end{lem}
\begin{proof}
We get by directly restraining the results of lemma \ref{lem:unionbound-adap}
to the factorizable measure jointly with probability at least $1-\epsilon$,

\[
\left.\begin{array}{c}
\xi\left(\lambda\right)\rho(\mathbb{P}D_{n}^{S}-D_{n}^{S})\\
\xi\left(\lambda\right)\rho(D_{n}^{S}-\mathbb{P}D_{n}^{S})
\end{array}\right\} \leq\xi\left[f(n,\lambda)\right]+\mathcal{K}\left(\rho,\pi_{\theta}\pi\right)+\mathcal{K}\left(\xi,\nu\right)+\log\frac{2}{\epsilon}.
\]

We rewrite the first part of the equation for any $\rho\in\mathcal{M}_{1}^{+}(\Theta\times\mathcal{X}^{n})$,
$\xi\in\mathcal{M}_{1}^{+}(\mathcal{I})$
\[
\rho\left[\mathbb{P}D_{n}^{S}\right]\leq\rho(D_{n}^{S})+\frac{1}{\xi(\lambda)}\left(\xi\left[f(n,\lambda)\right]+\mathcal{K}\left(\rho,\pi_{\theta}\pi\right)+\mathcal{K}\left(\xi,\nu\right)+\log\frac{2}{\epsilon}\right).
\]

One can take the best bound amongst possible measures,

\[
\rho_{\tilde{\xi}(\lambda)}\left[\mathbb{P}D_{n}^{S}\right]\leq\inf_{\xi\in\mathcal{M}_{1}^{+}(\mathcal{I})}\left\{ \inf_{\rho\in\mathcal{M}_{1}^{+}(\Theta\times\mathcal{X}^{n})}\left\{ \rho(D_{n}^{S})+\frac{1}{\xi(\lambda)}\mathcal{K}\left(\rho,\pi_{\theta}\pi\right)\right\} +\frac{1}{\xi(\lambda)}\left(\xi\left[f(n,\lambda)\right]+\mathcal{K}\left(\xi,\nu\right)+\log\frac{2}{\epsilon}\right)\right\} .
\]
 Note that the two infimum are computed sequentially, starting with
$\rho$, leading to a problem that can be computed for fixed $\xi$.
Then the infimum in $\xi$.

The infimum over $\rho$ is achieved in the exponential weight measure
with inverse temperature $\xi(\lambda)$ at $-\log Z_{\xi(\lambda)}$
by equations \ref{eq:var1} and \ref{eq:var2}. The measure $\tilde{\xi}$
is the minimum achieved by $-\log Z_{\xi(\lambda)}+\frac{1}{\xi(\lambda)}\left(\xi\left[f(n,\lambda)\right]+\mathcal{K}\left(\xi,\nu\right)+\log\frac{2}{\epsilon}\right)$.
By restricting the minimization to a specific class of probabilities
$\mathcal{F}$ we therefore get the algorithm described in definition
\ref{def:adap-ABC}, we plug-in the second part of lemma \ref{lem:unionbound-adap}
we get the result.
\end{proof}
As for the proof of lemma \ref{lem:bound} we use Assumption A\ref{assu:distance}
and Jensen's inequality to get with probability at least $1-\epsilon$,
\[
d(\rho_{\hat{\xi}(\lambda)}S,\mathbb{P}S)\leq\rho_{\hat{\xi}(\lambda)}\left[\mathbb{P}D_{n}^{S}\right]\leq\inf_{\xi\in\mathcal{F}}\left\{ \inf_{\rho\in\mathcal{M}_{1}^{+}(\Theta\times\mathcal{X}^{n})}\left\{ \rho(\mathbb{P}D_{n}^{S})+\frac{2}{\xi(\lambda)}\mathcal{K}\left(\rho,\pi_{\theta}\pi\right)\right\} +\frac{2}{\xi(\lambda)}\left(\xi\left[f(n,\lambda)\right]+\mathcal{K}\left(\xi,\nu\right)+\log\frac{2}{\epsilon}\right)\right\} .
\]
Again we put 
\[
\varrho_{\theta^{\star},\delta}(dX^{n},d\theta)=\frac{\pi_{\theta}(dX^{n})\pi(d\theta)\textbf{1}_{\Vert\theta-\theta^{\star}\Vert<\delta}}{\pi(\Vert\theta-\theta^{\star}\Vert<\delta)}.
\]
The rest of the proof goes along the line of the end of the proof
of \ref{lem:bound}.

\subsection{Proof of Section \ref{sec:Examples-of-bounds}}

\subsubsection{Proof of Lemma \ref{lem:McDiarmid1}}

To prove the result we will start be defining a vector of observations
with the $i$-th component replaced by $Y^{\prime},$ i.e. $Y^{n,\prime}=(Y_{1},\cdots,Y_{i-1},Y^{\prime},Y_{i},\cdots,Y_{n}).$
For this new sample we have by the triangle inequality 
\[
\left|d\left(S\left(Y^{n}\right),S\left(X^{n}\right)\right)-d\left(S\left(Y^{n,\prime}\right),S\left(X^{n}\right)\right)\right|\leq d\left(S\left(Y^{n,\prime}\right),S\left(Y^{n}\right)\right).
\]
 By the definition of $Y^{n,\prime}$ and $d$ the term on the right
hand-side is $d\left(S\left(Y^{n,\prime}\right),S\left(Y^{n}\right)\right)=\frac{1}{n}\left\Vert H(Y_{i})-H(Y^{\prime})\right\Vert _{p}\leq\frac{2m^{\frac{1}{p}}K}{n}$.
By the bounded difference inequality theorem 6.2 \cite{Boucheron2013}
we get the desired result

\subsubsection{Proof of Theorem \ref{thm:euclidean}}

We can apply the lemma \ref{lem:bound} to the framework of this section
then with probability at least $1-\epsilon$, 

\begin{multline*}
d(\varrho_{\lambda}^{ABC}(S),\mathbb{P}(H))\leq\inf_{\theta\in\Theta}d(\pi_{\theta}(H),\mathbb{P}(H))+\sup_{\theta:\Vert\theta -\theta^\star\Vert\leq \delta}\pi_{\theta}\left\{ d(S,\pi_{\theta}S)\right\} +\delta L+\mathbb{P}\left\{ d(S,\mathbb{P}S)\right\} \\
+\frac{2\lambda K^{2}m^{\frac{2}{p}}}{n}-\frac{2}{\lambda}\log\pi\left(\left\{ \Vert\theta-\theta^{\star}\Vert<\delta\right\} \right)+\frac{2}{\lambda}\log\frac{2}{\epsilon}.
\end{multline*}
We need to control $\mathbb{P}\left\{ d(S,\mathbb{P}S)\right\} $
and $\pi_{\theta^{\star}}\left\{ d(S,\pi_{\theta^{\star}}S)\right\} $,
we use Jensen for the first inequality, on the third line we use Nemirovki's
inequality (see \cite{Boucheron2013} p. 335), 
\begin{align*}
\mathbb{P}\left\{ d(S,\mathbb{P}S)\right\}  & =\mathbb{P}\left\{ \left\Vert \frac{1}{n}\sum_{i=1}^{n}H(Y_{i})-\mathbb{P}H\right\Vert _{p}\right\} \\
 & \leq\sqrt{\frac{1}{n^{2}}\mathbb{P}\left\{ \left\Vert \sum_{i=1}^{n}H(Y_{i})-\mathbb{P}H\right\Vert _{p}^{2}\right\} }\\
 & \leq\sqrt{\frac{K(p,m)}{n^{2}}\sum_{i=1}^{n}\mathbb{P}\left\{ \left\Vert H(Y_{i})-\mathbb{P}H\right\Vert _{p}^{2}\right\} },
\end{align*}
 where $K(p,m)=\min\left(m,p-1,2e\log m\right).$ Using the i.i.d.
hypothesis and the definition of $C$ we get,

\begin{align*}
\mathbb{P}\left\{ d(S,\mathbb{P}S)\right\}  & \leq\sqrt{\frac{K(p,m)}{n}\mathbb{P}\left\{ \left\Vert H(Y_{i})-\mathbb{P}H\right\Vert _{p}^{2}\right\} }\\
 & \leq\sqrt{\frac{K(p,m)m^{\frac{2}{p}}}{n}m^{2}\max_{j\leq m}\mathbb{P}\left\{ \left(h_{j}-\mathbb{P}h_{j}\right)^{2}\right\} }\\
 & \leq\frac{Cm^{\frac{1}{p}+1}}{\sqrt{n}}
\end{align*}

We get a similar bound for $\pi_{\theta}$, $\pi_{\theta}\left\{ d(S,\pi_{\theta}S)\right\} \leq\frac{Cm^{\frac{1}{p}+1}}{\sqrt{n}}$.

\subsubsection{Proof of Corollary \ref{cor:euclidean-Gaussian}}

We start by lower bounding the small probability,

\begin{align*}
\log\pi\left(\left\{ \Vert\theta-\theta^{\star}\Vert_{2}<\delta\right\} \right) & =\log\pi\left(\left\{ \theta:\sum_{i=1}^{d}(\theta_{i}-\theta_{i}^{\star})^{2}\leq\delta^{2}\right\} \right)\\
 & \geq d\min_{i}\log\pi\left(\left\{ \theta:(\theta-\theta_{i}^{\star})^{2}\leq\frac{\delta^{2}}{d}\right\} \right)\\
 & \geq d\min_{i}\log\int_{\frac{\theta_{i}^{\star}}{\sqrt{\vartheta}}-\frac{\delta}{\sqrt{\vartheta d}}}^{\frac{\theta_{i}^{\star}}{\sqrt{\vartheta}}+\frac{\delta}{\sqrt{\vartheta d}}}\Phi_{1}(dx;0,1)\\
 & \geq d\min_{i}\log\left(\frac{\delta}{2\sqrt{\vartheta d}}\varphi\left(\frac{\theta_{i}^{\star}}{\sqrt{\vartheta}}+\frac{\delta}{\sqrt{\vartheta d}}\right)\right)\\
 & =d\log\left(\frac{\delta}{2\sqrt{2\pi\vartheta d}}\exp\left[-\frac{1}{2}\left(\frac{1}{\sqrt{\vartheta}}+\frac{\delta}{\sqrt{\vartheta d}}\right)^{2}\right]\right)\\
 & \geq d\log\left\{ \frac{\delta}{2\sqrt{2\pi\vartheta d}}\exp\left(-\frac{1}{\vartheta}-\frac{\delta^{2}}{\vartheta d}\right)\right\} 
\end{align*}
 The bound can therefore be written, 
\begin{multline*}
\Vert\varrho_{\lambda}^{ABC}(H)-\mathbb{P}(H)\Vert_{p}\leq\inf_{\theta\in\Theta}\Vert\pi_{\theta}(H)-\mathbb{P}(H)\Vert_{p}+\frac{Cm^{\frac{1}{p}+1}}{\sqrt{n}}+L\delta+\frac{2\lambda K^{2}m^{\frac{2}{p}}}{n}+\frac{2d}{\lambda}\left\{ \log\frac{2\sqrt{2\pi\vartheta d}}{\delta}+\frac{1}{\vartheta}+\frac{\delta^{2}}{\vartheta d}\right\} +\frac{2}{\lambda}\log\frac{2}{\epsilon}.
\end{multline*}
 We get the result by plugging$\lambda=\sqrt{\frac{dn}{K^{2}m^{\frac{2}{p}}}}$
and $\delta=\sqrt{\frac{\vartheta}{n}}$. 

\subsubsection{Proof of Corollary \ref{cor:adap-gaussian } }

We start by the result of Theorem \ref{thm:Adap}, 
\begin{multline*}
\left\Vert \varrho_{\epsilon,\hat{\xi}}(H)-\mathbb{P}H\right\Vert _{p}\leq\inf_{\theta\in\Theta}\left\Vert \pi_{\theta}(S)-\mathbb{P}(S)\right\Vert _{p}+\frac{2Cm^{\frac{1}{p}}}{\sqrt{n}}e^{L\delta^{\alpha}}\\
+\inf_{\xi\in\mathcal{F}}\left[\frac{1}{\xi(\lambda)}\left\{ 2\xi\left[\lambda^{2}\right]\frac{K^{2}m^{\frac{2}{p}}}{n}-2\log\pi\left(\left\{ \Vert\theta-\theta^{\star}\Vert<\delta\right\} \right)+2\mathcal{K}(\xi,\nu)+2\log\frac{2}{\epsilon}\right\} \right],
\end{multline*}
 we need to compute the Kullback-Leibler term and the first and second
order moments under an exponential distribution. We have,

\begin{align*}
\mathcal{K}(\xi,\nu) & =\log\frac{\beta}{\alpha}+\frac{\alpha-\beta}{\beta},\\
\xi(\lambda) & =\frac{1}{\beta},\\
\xi(\lambda^{2}) & =\frac{2}{\beta^{2}}.
\end{align*}
 We plug those results in the above equation, 
\begin{multline*}
\left\Vert \varrho_{\epsilon,\hat{\xi}}(H)-\mathbb{P}H\right\Vert _{p}\leq\inf_{\theta\in\Theta}\left\Vert \pi_{\theta}(S)-\mathbb{P}(S)\right\Vert _{p}+2\frac{Cm^{\frac{1}{p}+1}}{\sqrt{n}}+L\delta\\
+\inf_{\beta\in\mathbb{R}_{+}}\left[\beta\left\{ \frac{4}{\beta^{2}}\frac{K^{2}m^{\frac{2}{p}}}{n}-2\log\pi\left(\left\{ \Vert\theta-\theta^{\star}\Vert<\delta\right\} \right)+2\log\frac{\beta}{\alpha}+2\frac{\alpha-\beta}{\beta}+2\log\frac{2}{\epsilon}\right\} \right].
\end{multline*}
 Using the same bound as in corollary \ref{cor:euclidean-Gaussian}
for the small ball under the prior we get,
\begin{multline*}
\left\Vert \varrho_{\epsilon,\hat{\xi}}(H)-\mathbb{P}H\right\Vert _{p}\leq\inf_{\theta\in\Theta}\left\Vert \pi_{\theta}(S)-\mathbb{P}(S)\right\Vert _{p}+2\frac{Cm^{\frac{1}{p}+1}}{\sqrt{n}}+L\delta\\
+\inf_{\beta\in\mathbb{R}_{+}}\left[\frac{4}{\beta}\frac{K^{2}m^{\frac{2}{p}}}{n}+2d\beta\left\{ \log\frac{2\sqrt{2\pi\vartheta d}}{\delta}+\frac{1}{\vartheta}+\frac{\delta^{2}}{\vartheta d}\right\} +2\beta\log\frac{\beta}{\alpha}+2\left(\alpha-\beta\right)+2\beta\log\frac{2}{\epsilon}\right].
\end{multline*}
 Using the fact that $\alpha<\beta$ otherwise the Kullback-Leibler
does not exist, and putting $\beta=\sqrt{\frac{K^{2}m^{\frac{2}{p}}}{nd}}$,
we get 

\begin{multline*}
\left\Vert \varrho_{\epsilon,\hat{\xi}}(H)-\mathbb{P}H\right\Vert _{p}\leq\inf_{\theta\in\Theta}\left\Vert \pi_{\theta}(S)-\mathbb{P}(S)\right\Vert _{p}+2\frac{Cm^{\frac{1}{p}+1}}{\sqrt{n}} +L\delta+4Km^{\frac{1}{p}}\sqrt{\frac{d}{n}}+2Km^{\frac{1}{p}}\sqrt{\frac{d}{n}}\left\{ \log\frac{2\sqrt{2\pi\vartheta d}}{\delta}+\frac{1}{\vartheta}+\frac{\delta^{2}}{\vartheta d}\right\} \\
+\frac{Km^{\frac{1}{p}}}{\sqrt{nd}}\log\left(\frac{K^{2}m^{\frac{2}{p}}}{nd\alpha}\right)+2\frac{Km^{\frac{1}{p}}}{\sqrt{nd}}\log\frac{2}{\epsilon}.
\end{multline*}
Put $\delta=\sqrt{\frac{\vartheta}{n}}$ to get the result.

\subsubsection{Proof of Lemma \ref{lem:McDiarmid-nonparam}}

As for the proof of lemma \ref{lem:McDiarmid1} define a sample $Y^{\prime,n}=\left(Y_{1},\cdots,Y^{\prime},\cdots Y_{n}\right)$
where the $i$-th sample has been replaced by $Y^{\prime}$ with the
same distribution. By the triangular inequality, 
\[
\left|d(S(X^{n}),S(Y^{n}))-d(S(X^{n}),S(Y^{n,\prime}))\right|\leq d(S(Y^{n}),S(Y^{n,\prime})),
\]
 hence in this case $d(S(Y^{n}),S(Y^{n,\prime}))=\sqrt{\frac{1}{n^{2}}(Y_{i}-Y^{\prime})}\leq\sqrt{\frac{2K}{n^{2}}}.$
Hence by applying the bounded difference inequality \cite{Boucheron2013}
we get the result.

\subsubsection{Proof of Lemma \ref{lem:small-ball}}

From Lemma 5.3. of \cite{van2008reproducing} we get that the non-centered
Gaussian small ball probability is characterized by its concentration
function, i.e. for any $f^{\star}$ in the support, 
\[
-\log\pi\left[\left\Vert f-f^{\star}\right\Vert <\delta\right]\leq\inf_{h\in\mathbb{H}:\left\Vert h-f^{\star}\right\Vert \leq\frac{\delta}{2}}\frac{1}{2}\left\Vert h\right\Vert _{\mathbb{H}}^{2}-\log\pi\left[\left\Vert f\right\Vert <\frac{\delta}{2}\right],
\]
 where $\left(\mathbb{H},\left\Vert .\right\Vert _{\mathbb{H}}\right)$
is the reproducing kernel Hilbert space of the Gaussian variable $f$.
\cite{van2007bayesian} give a bound for the two quantities in the
right hand-side, in the case where the Gaussian process has a spectral
measure with exponentially decreasing tails and $f^{\star}\in\mathcal{H}^{\beta}\left(\left[0,1\right]\right)$.
The condition on the spectral measure is satisfied in particular for
the centered Gaussian process with Gaussian kernel. Suppose that the
Gaussian process is re scaled with parameter $c\leq1$ then by Theorem
2.4 \cite{van2007bayesian} there exists $\delta_{0}>0$ and a constant
$K$ such that for any $\delta\in(0,\delta_{0})$ the centered small
ball probability satisfies,

\[
-\log\pi\left[\left\Vert f\right\Vert <2\delta\right]\leq\frac{K}{c}\left(\log\frac{1}{c\delta^{2}}\right)^{2}.
\]
 Lemma 2.2 of the same paper gives a bound for the second part of
the concentration function, under the assumption on $\pi$ and $f^{\star}$
there exist constants $D_{f^{\star}}$ and $C_{f^{\star}}$ depending
only on $f^{\star}$ such that 
\[
\inf_{h\in\mathbb{H}:\left\Vert h-f^{\star}\right\Vert \leq C_{f^{\star}}c^{\beta}}\frac{1}{2}\left\Vert h\right\Vert _{\mathbb{H}}^{2}\leq D_{f^{\star}}\left(\frac{1}{c}\right).
\]
 Choosing $c$ such that $c^{\beta}\leq\delta$, we get the result
by combining the two bounds. 

\subsubsection{Proof of Theorem \ref{thm:nonpara}}

We start with lemma \ref{lem:bound} applied to the framework of this
section, with probability at least $1-\epsilon$, 

\begin{multline*}
d(\varrho_{\lambda}^{ABC}(S),\mathbb{P}(H))\leq\inf_{\theta\in\Theta}d(\pi_{\theta}(H),\mathbb{P}(H))+\sup_{\theta:\Vert\theta-\theta^\star\Vert}\pi_{\theta}\left\{ d(S,\pi_{\theta}S)\right\}+ L\delta+\mathbb{P}\left\{ d(S,\mathbb{P}S)\right\} \\
+\frac{\lambda K}{n}-\frac{2}{\lambda}\log\pi\left(\left\{ \Vert\theta-\theta^{\star}\Vert<\delta\right\} \right)+\frac{2}{\lambda}\log\frac{2}{\epsilon}.
\end{multline*}
 We start as before to bound $\pi_{\theta}\left\{ d(S,\pi_{\theta}S)\right\} $
and $\mathbb{P}\left\{ d(S,\mathbb{P}S)\right\} $, 
\begin{align*}
\pi_{\theta}\left\{ d(S,\pi_{\theta}S)\right\}  & =\pi_{\theta}\left\{ \sqrt{\frac{1}{n^{2}}\sum_{i=1}^{n}\left[X_{i}-\pi_{\theta}\left(X_{i}\right)\right]{}^{2}}\right\} \\
 & \leq\sqrt{\pi_{\theta}\left\{ \frac{1}{n^{2}}\sum_{i=1}^{n}\left[X_{i}-\pi_{\theta}\left(X_{i}\right)\right]{}^{2}\right\} }\\
 & \leq\frac{1}{\sqrt{n}}\sup_{i}\sqrt{\mathbb{\mathbb{\pi_{\theta}}}\left\{ \left(X_{i}-\pi_{\theta}(X_{i})\right)^{2}\right\} }.
\end{align*}
 We get a similar bound for $\mathbb{P}\left\{ d(S,\mathbb{P}S)\right\} $,
hence 
\begin{multline*}
d(\varrho_{\lambda}^{ABC}(S),\mathbb{P}(H))\leq\inf_{\theta\in\Theta}d(\pi_{\theta}(H),\mathbb{P}(H))+\frac{1}{\sqrt{n}}C+\frac{\lambda K}{n}\\
-\frac{2}{\lambda}\log\pi\left(\left\{ \Vert\theta-\theta^{\star}\Vert<\delta\right\} \right)+\frac{2}{\lambda}\log\frac{2}{\epsilon}.
\end{multline*}

From lemma \ref{lem:small-ball} we get an estimate of the prior concentration,
we get that for $\delta_{n}\geq c_{n}^{\beta},$ 
\begin{multline*}
d(\varrho_{\lambda}^{ABC}(S),\mathbb{P}(H))\leq\inf_{\theta\in\Theta}d(\pi_{\theta}(H),\mathbb{P}(H))+\frac{1}{\sqrt{n}}C+L\delta+\frac{\lambda K}{n}\\
+\frac{2}{\lambda}\left[D_{0}\left(\frac{1}{c_{n}}\right)+C_{0}\frac{1}{c_{n}}\left(\log\frac{1}{c_{n}\delta_{n}}\right)^{2}\right]+\frac{2}{\lambda}\log\frac{2}{\epsilon}.
\end{multline*}

Now we put $\delta_{n}\asymp\left(\frac{\log^{2}n}{n}\right)^{\frac{\beta}{2\beta+1}}$,
$c_{n}\asymp\left(\frac{\log^{2}n}{n}\right)^{\frac{1}{2\beta+1}}$
and $\lambda_{n}\asymp n^{\frac{\beta+1}{2\beta+1}}\left(\log n\right)^{\frac{\beta}{2\beta+1}}$
to get the result for $n$ sufficiently large as to ensure $\delta_{n}<\bar{\delta}$. 

\section{Summary }

We have explored convergence results for the ABC algorithm in the
specific case of the exponential kernel. This kernel is introduced
for technical reasons, in particular because it is the solution of
a variational problem. The results in the paper suggest that ABC can
be used in a mispecified scenario (i.e. $\mathbb{P}\notin\left\{ \pi_{\theta},\theta\in\Theta\right\} $
in the notations of the paper), at the cost of choosing a larger window
in the kernel. In particular it is instructive to note that we do
not want this parameter to go to zero too fast even if it was possible
computationally. We obtain oracle inequalities for the expected statistics
under the ABC distribution. We show that they can be extended in some
cases to oracle inequalities in the parameter space. The results rely
mostly on the exponential concentration of the distance and some regularity
of the model around the oracle parameter. One could remove the need
for concentration inequalities on this problem by using the techniques
introduced in \cite{grunwald2016,mendelson2014,Alquier2017,Bhattacharya2017},
we leave this for a future study. A nice aspect of the result is that
they are given for finite sample sizes and in deviation.  

We also showed that we can obtain empirical bounds and adaptive oracle
inequalities in the bandwidth. The proposed bounds can be used to
gain intuition on the distance to choose and the size of the summary
statistics for a given problem. 

Finally we suggest some methodological improvement to the previously
known SMC-ABC of \cite{DelMoral2012}, allowing for further adaptation.
We would like to stress at this point that, although we believe that
SMC can perform well on this kind of problems, it is by no means the
only approach to sample from the pseudo distribution.

\subsection*{Acknowledgments}

I would like to warmly thank Pierre Alquier and Nicolas Chopin for
the helpful discussions and comments.

\bibliographystyle{plainnat}
\bibliography{/home/james/Dropbox/biball}

\appendix

\section{Implementation Details}

\label{sec:appendix-implementation}

We describe some building blocks of the algorithms of Section \ref{sec:Monte-Carlo-algorithm}

\begin{algorithm}
\caption{Systematic resampling\label{alg:SystResampling}}

\begin{description} \item[Input:] Normalised weights  
$W_t^j:= w_t(\theta_{t-1}^j)/\sum_{i=1}^N w_t(\theta_{t-1}^i)$. 
\item[Output:] indices $A^i\in\{1,\ldots,N\}$, for $i=1,\ldots,N$.  \item[a.] Sample $U\sim \mathcal{U}nif_{[0,1]}$. \item[b.] Compute cumulative weights as %$c^0=0$,  $C^n=\sum_{m=1}^n NW^m$.
\item[c.] Set $s\leftarrow U$, $m\leftarrow 1$.  
\item[d.] \textbf{For} $n= 1:N$ \item[]$\qquad$ \textbf{While} $C^m<s$ \textbf{do} $m\leftarrow m+1$.  \item[]$\qquad$ $A^n\leftarrow m$, and $s\leftarrow s+1$.  \item[]$\quad$ \textbf{End For} \end{description} 
\end{algorithm}

\begin{tabular}{|c|c|c|}

	\hline

	 & In common  & Tested\\

	\hline

	 Figure \ref{fig:boxplot} &  \begin{minipage}{6cm}Algorithm are all versions \\of \cite{DelMoral2012}\end{minipage} & \begin{minipage}{6cm} One algorithm with exponential weights\\ the other with uniform weigths\end{minipage}\\

	\hline
	 Figure \ref{fig:accept} & \begin{minipage}{6cm} Algorithm are all versions of \cite{DelMoral2012}\\ with exponential weights\end{minipage} &\begin{minipage}{6cm}On has adaptive choice of $M$ \\ not the other \end{minipage}\\

	\hline
	 Figure \ref{fig:MSE} &  \begin{minipage}{6cm}Algorithm are all versions of \cite{DelMoral2012} \\ with exponential weights and adaptive choice  $M$\end{minipage}&\begin{minipage}{6cm} One with fixed temperature, one with adaptive temperature and one is the algorithm of  \cite{DelMoral2012} for a benchmark\end{minipage}\\

	\hline

\end{tabular}
\end{document}